\documentclass{amsart}
\usepackage{amsmath,amscd,amssymb,amsthm,amsbsy,amscd,stmaryrd}
\usepackage{mathrsfs}
\usepackage {hyperref}

\newtheorem{theorem}{Theorem}[section]
\newtheorem{lemma}[theorem]{Lemma}
\newtheorem{proposition}[theorem]{Proposition}
\newtheorem{corollary}[theorem]{Corollary}

\theoremstyle{definition}
\newtheorem{definition}[theorem]{Definition}

\numberwithin{equation}{section}

\begin{document}
\title[Generalized Ricci flow I]
{Generalized Ricci flow I: Higher derivatives estimates for compact
manifolds}
\author{Yi Li}
\address{Department of Mathematics, Harvard University, One Oxford street, Cambridge, MA 02138}
\email{yili@math.harvard.edu}



\begin{abstract} We consider a generalized Ricci flow
with a given (not necessarily closed) three-form and establish
the higher derivatives estimates for compact manifolds. As an
application, we prove the compactness theorem for this generalized
Ricci flow. The similar results still hold for a more generalized Ricci flow.
\end{abstract}


\maketitle

\tableofcontents

\section{Introduction}

Throughout this paper manifolds always mean smooth and
closed (compact and without boundary) manifolds. Let $\mathfrak{Met}(M)$ denote
the space of smooth metrics on a manifold $M$, and $C^{\infty}(M)$ the set of
all smooth functions on $M$. We denote by $C$ the universal constants
depending only on the dimension of $M$,
which may take different values at different places.

An important and natural problem in differential geometry is to find a canonical metric on a given manifold. A classical example is the uniformization theorem (e.g., \cite{CK}), which
says that every smooth surface admits a unique conformal metric of constant
curvature. To generalize to higher dimensional manifolds, Hamilton \cite{H1}
introduced a system of equations
\begin{equation}
\frac{\partial g_{ij}}{\partial t}=-2R_{ij},\label{1.1}
\end{equation}
now called the Ricci flow, an analogue of the heat equation for metrics.

There are two ways to understand the Ricci flow: one way comes from the two-dimensional
sigma model (see \cite{B}), while another comes from Perelman's energy
functional (\cite{P}) defined by
\begin{equation}
\mathcal{F}(g,f)=\int_{M}\left(R+|\nabla f|^{2}\right)e^{-f}dV_{g}, \ \ \
(g,f)\in\mathfrak{Met}(M)\times C^{\infty}(M),\label{1.2}
\end{equation}
where $R$, $\nabla$, and $dV_{g}$, is the scalar curvature, Levi-Civita connection, and
volume form of $g$, respectively. He
showed that the Ricci flow is the gradient flow of (\ref{1.2}) and the
functional $\mathcal{F}$ is monotonic along this gradient flow. Precisely, under
the following system
\begin{equation}
\frac{\partial g_{ij}}{\partial t}=-2R_{ij}, \ \ \
\frac{\partial f}{\partial t}=-R-\Delta f+|\nabla f|^{2},\label{1.3}
\end{equation}
we have
\begin{equation}
\frac{d}{dt}\mathcal{F}(g,f)=2\int_{M}\left|R_{ij}+\nabla_{i}\nabla_{j}f\right|^{2}
e^{-f}dV_{g}\geq0.\label{1.4}
\end{equation}
Perelman's energy functional plays an essential role in determining the structures
of singularities of the Ricci flow and then the proof of
Poincar\'e conjecture and Thurston's generalization conjecture; for more details we refer the readers to
\cite{CZ, CLN, CCGGIIKLLN, KL, MT, P}.

\subsection{Ricci flow coupled with a one form or a two form}

If we consider the two-dimensional nonlinear sigma model \cite{B, OSW}, then we obtain a
generalized Ricci flow that is the Ricci flow coupled with the evolution
equation for a two form. This flow can be also obtained from the point of view
of Perelman-type energy functional.

Denoting by $\mathcal{A}^{p}(M)$ the space of $p$-forms on $M$, we consider the energy functional
\begin{equation*}
\mathcal{F}^{(1)}: \mathfrak{Met}(M)\times\mathcal{A}^{2}(M)\times C^{\infty}(M)\longrightarrow\mathbb{R}
\end{equation*}
defined by
\begin{equation}
\mathcal{F}^{(1)}(g,B,f)=\int_{M}\left(R+|\nabla f|^{2}-\frac{1}{12}|H|^{2}
\right)e^{-f}dV_{g},\label{1.5}
\end{equation}
where $H=dB$. As showed in \cite{OSW}, the gradient flow of $\mathcal{F}^{(1)}$
satisfies
\begin{eqnarray}
\frac{\partial g_{ij}}{\partial t}&=&-2R_{ij}-2\nabla_{i}\nabla_{j}f+\frac{1}{2}
H_{i}{}^{k\ell}H_{jk\ell},\label{1.6}\\
\frac{\partial B_{ij}}{\partial t}&=&3\nabla_{k}H^{k}{}_{ij}
-3H^{k}{}_{ij}\nabla_{k}f,\label{1.7}\\
\frac{\partial f}{\partial t}&=&-R-\Delta f+\frac{1}{4}|H|^{2},\label{1.8}
\end{eqnarray}
and under a family of diffeomorphisms the system (\ref{1.6})---(\ref{1.8}) is
equivalent to
\begin{eqnarray}
\frac{\partial g_{ij}}{\partial t}&=&-2R_{ij}+\frac{1}{2}H_{i}{}^{k\ell}
H_{jk\ell},\label{1.9}\\
\frac{\partial B_{ij}}{\partial t}&=&3\nabla_{k}H^{k}{}_{ij},\label{1.10}\\
\frac{\partial f}{\partial t}&=&-R-\Delta f+|\nabla f|^{2}+\frac{1}{4}|H|^{2}.\label{1.11}
\end{eqnarray}

Using the adjoint operator $d^{\ast}$, the equation (\ref{1.10})
can be written as
\begin{equation}
\frac{\partial B_{ij}}{\partial t}=-(d^{\ast}H)_{ij},\label{1.12}
\end{equation}
and therefore (because of $H=dB$)
\begin{equation}
\frac{\partial H}{\partial t}=-dd^{\ast}H=\Delta_{{\rm HL}}H,\label{1.13}
\end{equation}
where $\Delta_{{\rm HL}}=-(dd^{\ast}+d^{\ast}d)$ denotes the
Hodge-Laplace operator.

The flow (\ref{1.9})---(\ref{1.10}) can be interpreted as the connection
Ricci flow \cite{S1}. If we replace $H=dB$ by $F=dA$, i.e.,
replace a two form by a one form, then the flow (\ref{1.6})---(\ref{1.7})
or (\ref{1.9})---(\ref{1.10}) is exactly the Ricci Yang-Mills flow
studied by Streets \cite{S2} and Young \cite{Y}.

\subsection{Ricci flow coupled with a one form and a two form}

There is another generalized Ricci flow which connects to Thurston's
conjecture---roughly stating that a three-dimensional manifold with a given
topology has a canonical decomposition into simple three-dimensional
manifolds, each of which admits one, and only one, of eight homogeneous
geometries: $\mathbb{S}^{3}$, the round three-sphere; $\mathbb{R}^{3}$, the Euclidean space; $\mathbb{H}^{3}$, the standard hyperbolic space; $\mathbb{S}^{2}\times\mathbb{R}$;
$\mathbb{H}^{2}\times\mathbb{R}$; ${\rm Nil}$, the three-dimensional nilpotent
Heisenberg group; $\widetilde{{\rm SL}}(2,\mathbb{R})$; ${\rm Sol}$, the three
-dimensional solvable Lie group. The proof of Thurston's conjecture can be found in \cite{CZ, KL, MT, P}.

To better understanding Thurston's conjecture, Gegenberg and Kunstatter
\cite{GK} proposed a generalized flow by considering the modified
$3D$ stringy theory. This flow is the Ricci flow coupled with evolution
equations for a one form and a two form. As in (\ref{1.5}), we define an energy
functional
\begin{equation*}
\mathcal{F}^{(2)}: \mathfrak{Met}(M)\times\mathcal{A}^{1}(M)
\times\mathcal{A}^{2}(M)\times C^{\infty}(M)\longrightarrow \mathbb{R}
\end{equation*}
by
\begin{equation}
\mathcal{F}^{(2)}(g,A,B,f)=\int_{M}
\left(R+|\nabla f|^{2}-\frac{1}{12}|H|^{2}-\frac{1}{2}
|F|^{2}\right)e^{-f}dV_{g},\label{1.14}
\end{equation}
where $H=dB$, and $F=dA$. In \cite{HHKL}, the authors showed that the gradient flow of $\mathcal{F}^{(2)}$ satisfies
\begin{eqnarray}
\frac{\partial g_{ij}}{\partial t}&=&-2R_{ij}-2\nabla_{i}\nabla_{j}f+\frac{1}{2}
H_{i}{}^{k\ell}H_{jk\ell}+2F_{i}{}^{k}F_{jk},\label{1.15}\\
\frac{\partial A_{i}}{\partial t}&=&2\nabla_{j}F^{j}{}_{i}
-2F^{j}{}_{i}\nabla_{j}f,\label{1.16}\\
\frac{\partial B_{ij}}{\partial t}&=&3\nabla_{k}H^{k}{}_{ij}
-3H^{k}{}_{ij}\nabla_{k}f,\label{1.17}\\
\frac{\partial f}{\partial t}&=&-R
-\Delta f+\frac{1}{4}|H|^{2}+|F|^{2},\label{1.18}
\end{eqnarray}
and under a family of diffeomorphisms the system (\ref{1.15})---(\ref{1.18}) is
equivalent to
\begin{eqnarray}
\frac{\partial g_{ij}}{\partial t}&=&-2R_{ij}+\frac{1}{2}H_{i}{}^{k\ell}H_{jk\ell}
+2F_{i}{}^{k}F_{jk},\label{1.19}\\
\frac{\partial A_{i}}{\partial t}&=&2\nabla_{j}F^{j}{}_{i},\label{1.20}\\
\frac{\partial B_{ij}}{\partial t}&=&3\nabla_{k}H^{k}{}_{ij},\label{1.21}\\
\frac{\partial f}{\partial t}&=&-R-\Delta f+|\nabla f|^{2}+\frac{1}{4}|H|^{2}
+|F|^{2}.\label{1.22}
\end{eqnarray}

Using again the adjoint operator $d^{\ast}$, we have
\begin{equation}
\frac{\partial F}{\partial t}=\Delta_{{\rm HL}}F, \ \ \ \frac{\partial H}{\partial t}
=\Delta_{{\rm HL}}H.\label{1.23}
\end{equation}

The flow (\ref{1.19})---(\ref{1.21})
clearly contains the Ricci flow, the flow (\ref{1.9})---(\ref{1.10}) or the
connection Ricci flow, and the Ricci Yang-Mills flow; we expect this flow
can give another proof of the Poincar\'e conjecture and Thurston's
generalization conjecture, with less analysis on singularities.

\subsection{Main results}

For convenience,
we refer to {\rm GRF} the generalized Ricci flow and ${\rm RF}(A,B)$ the Ricci flow coupled
with a one form $A$ and a two form $B$.

Let $(M,g)$ denote an $n$-dimensional closed Riemannian
manifold with a three form $H=\{H_{ijk}\}$. In the first part of this paper we consider
the following GRF on $M$:
\begin{eqnarray}
\frac{\partial}{\partial
t}g_{ij}(x,t)&=&-2R_{ij}(x,t)+\frac{1}{2}H_{ik\ell}(x,t)H_{j}{}^{k\ell}(x,t),
\label{1.24}\\
\frac{\partial}{\partial t}H(x,t)&=&\Delta_{{\rm HL},g(x,t)}H(x,t), \ \ \
H(x,0)=H(x), \ g(x,0)=g(x).\label{1.25}
\end{eqnarray}

It is clearly from (\ref{1.9}) and (\ref{1.13}) that the gradient
flow of the energy functional $\mathcal{F}^{(1)}$ is a special case of (\ref{1.24})
---(\ref{1.25}). The corresponding case that $H$ is closed is called the refined generalized
Ricci flow (RGRF):
\begin{eqnarray}
\frac{\partial}{\partial
t}g_{ij}(x,t)&=&-2R_{ij}(x,t)+\frac{1}{2}H_{ik\ell}(x,t)H_{j}{}^{k\ell}(x,t),\label{1.26}\\
\frac{\partial}{\partial
t}H(x,t)&=&-dd^{\ast}_{g(x,t)}H(x,t), \ H(x,0)=H(x), \
g(x,0)=g(x).\label{1.27}
\end{eqnarray}
Here $d^{\ast}_{g(x,t)}$ is the dual operator of $d$ with respect to
the metric $g(x,t)$.

\begin{lemma} \label{l1.1}Under {\rm RGRF}, $H(x,t)$ is closed if the initial value $H(x)$ is closed.
\end{lemma}

\begin{proof} Since the exterior derivative $d$ is independent of
the metric, we have
\begin{equation*}
\frac{\partial}{\partial t}dH(x,t)=d\frac{\partial}{\partial
t}H(x,t)=d\left(-dd^{\ast}_{g(x,t)}H(x,t)\right)=0.
\end{equation*}
so $dH(x,t)=dH(x)=0$.
\end{proof}

The closedness of $H$ is very important and has physical
interpretation \cite{B, OSW}. Streets \cite{S1} considered the connection Ricci flow in which $H$ is the
geometric torsion of connection.

\begin{proposition} \label{p1.2}If $(g(x,t),H(x,t))$ is a solution of {\rm
RGRF} and the initial value $H(x)$ is closed, then it is also a
solution of {\rm GRF}.
\end{proposition}

\begin{proof} From Lemma \ref{l1.1} and the assumption we know that $H(x,t)$ are all closed.
Hence $\Delta_{{\rm HL},g(x,t)}H(x,t)=-dd^{\ast}_{g(x,t)}H(x,t)$.
\end{proof}

For {\rm GRF}, a basic and natural question is the existence.
The short-time existence for {\rm RGRF} has been established in \cite{HHKL},
where the authors have already showed the short-time existence
for ${\rm RF}(A,B)$ obviously including RGRF. In this paper, we prove the short-
time existence for RGF.

\begin{theorem} \label{t1.3}There is a unique solution to {\rm GRF} for
a short time.
More precisely, let $(M,g_{ij}(x))$ be an $n$-dimensional closed
Riemannian manifold with a three form $H=\{H_{ijk}\}$, then there
exists a constant $T=T(n)>0$ depending only on $n$ such that the
evolution system
\begin{eqnarray*}
\frac{\partial}{\partial
t}g_{ij}(x,t)&=&-2R_{ij}(x,t)+\frac{1}{2}g^{kp}(x,t)g^{\ell q}(x,t)
H_{ik\ell}(x,t)H_{jpq}(x,t),
\\
\frac{\partial}{\partial t}H(x,t)&=&\Delta_{{\rm HL},g(x,t)}H(x,t), \ \ \
H(x,0)=H(x), \ \ \ g(x,0)=g(x),
\end{eqnarray*}
has a unique solution $(g_{ij}(x,t),H_{ijk}(x,t))$ for a short time
$0\leq t\leq T$.
\end{theorem}

After establishing the local existence, we are able to prove the
higher derivatives estimates for {\rm GRF}. Precisely, we have the
following

\begin{theorem} \label{t1.4}Suppose that $(g(x,t),H(x,t))$ is a solution to {\rm GRF} on a
closed manifold $M^{n}$ and $K$ is an arbitrary given positive
constant. Then for each $\alpha>0$ and each integer $m\geq1$ there
exists a constant $C_{m}$ depending on $m, n, \max\{\alpha,1\}$, and
$K$ such that if
\begin{equation*}
|{\rm Rm}(x,t)|_{g(x,t)}\leq K, \ \ \ |H(x)|_{g(x)}\leq K
\end{equation*}
for all $x\in M$ and $t\in[0,\alpha/K]$, then
\begin{equation}
|\nabla^{m-1} {\rm Rm}(x,t)|_{g(x,t)}+|\nabla^{m}
H(x,t)|_{g(x,t)}\leq\frac{C_{m}}{t^{m/2}}\label{1.28}
\end{equation}
for all $x\in M$ and $t\in(0,\alpha/K]$.
\end{theorem}

As an application, we can prove the compactness theorem for {\rm
GRF}.

\begin{theorem} \label{t1.5}{\bf (Compactness for GRF)} Let
$\{(M_{k},g_{k}(t),H_{k}(t),O_{k})\}_{k\in\mathbb{N}}$ be a sequence of
complete pointed solutions to {\rm GRF} for
$t\in[\alpha,\omega)\ni0$ such that

\begin{itemize}

\item[(i)] there is a constant $C_{0}<\infty$ independent of
$k$ such that
\begin{equation*}
\sup_{(x,t)\in M_{k}\times(\alpha,\omega)}\left|{\rm
Rm}_{g_{k}(x,t)}\right|_{g_{k}(x,t)}\leq C_{0}, \ \ \ \sup_{x\in
M_{k}}|H_{k}(x,\alpha)|_{g_{k}(x,\alpha)}\leq C_{0},
\end{equation*}

\item[(ii)] there exists a constant $\iota_{0}>0$ satisfies
\begin{equation*}
{\rm inj}_{g_{k}(0)}(O_{k})\geq\iota_{0}.
\end{equation*}

\end{itemize}
Then there exists a subsequence $\{j_{k}\}_{k\in\mathbb{N}}$ such that
\begin{equation*}
(M_{j_{k}},g_{j_{k}}(t),H_{j_{k}}(t),O_{j_{k}})\longrightarrow
(M_{\infty},g_{\infty}(t),H_{\infty}(t),O_{\infty}),
\end{equation*}
converges to a complete pointed solution
$(M_{\infty},g_{\infty}(t),H_{\infty}(t),O_{\infty}),
t\in[\alpha,\omega)$ to {\rm GRF}
 as $k\to\infty$.
\end{theorem}

In the second part of this paper, we consider the Ricci flow coupled with
a one form and a two form. This flow is the gradient flow of
$\mathcal{F}^{(2)}$ and takes the form:
\begin{eqnarray}
\frac{\partial}{\partial t}g_{ij}(x,t)&=&-2R_{ij}+\frac{1}{2}
H_{i}{}^{k\ell}(x,t)H_{jk\ell}(x,t)+2F_{i}{}^{k}(x,t)F_{jk}(x,t),\label{1.29}\\
\frac{\partial}{\partial t}A_{i}(x,t)&=&2\nabla_{j}F^{j}{}_{i}(x,t), \ \ \ A_{i}(x,0)=A_{i}(x), \ \ \ g_{ij}(x,0)=g_{ij}(x),\label{1.30}\\
\frac{\partial }{\partial t}B_{ij}(x,t)&=&3\nabla_{k}H^{k}{}_{ij}(x,t), \ \ \
B_{ij}(x,0)=B_{ij}(x).\label{1.31}
\end{eqnarray}
Here $A=\{A_{i}\}$ and $B=\{B_{ij}\}$ is a one form and a two form
on $M$, respectively, and $F=dA, H=dB$. For this flow, we can also prove the short-time existence, higher derivative estimates, and the compactness theorem.

The rest of this paper is organized as follows. In Section \ref{2}, we prove
the short-time existence and uniqueness of the GRF for any given
three form $H$. In Section \ref{3}, we compute the evolution equations for the
Levi-Civita connections, Riemann, Ricci, and scalar curvatures of a
solution to the GRF. In Section \ref{4}, we establish higher derivative
estimates for GRF, called {\it Bernstein-Bando-Shi(BBS) derivative
estimates} (e.g., \cite{CZ, CK, CCGGIIKLLN, MT, Shi}). In Section \ref{5}, we prove the compactness
theorem for GRF by using BBS estimates. In Section \ref{6}, based on the work
of \cite{HHKL}, the similar results are established for ${\rm RF}(A,B)$.

{\bf Acknowledgment.} The author thanks his advisor, Professor
Shing-Tung Yau, for helpful discussions. The author expresses his
gratitude to Professor Kefeng Liu for his interest in this work and
for his numerous help in mathematics. He also thanks Valentino Tosatti
and Jeff Streets for several useful conversations.

\section{Short-time existence of {\rm GRF}}\label{2}

In this section we establish the short-time existence for GRF. Our
method is standard, that is, DeTurck trick which is used in Ricci
flow to prove its short-time existence. We assume that $M$ is an
$n$-dimensional closed Riemannian manifold with metric
\begin{equation}
d\widetilde{s}^{2}=\widetilde{g}_{ij}(x)dx^{i}dx^{j}\label{2.1}
\end{equation}
and with Riemannian curvature tensor $\{\widetilde{R}_{ijk\ell}\}$. We
also assume that $\widetilde{H}=\{\widetilde{H}_{ijk}\}$ is a fixed
three form on $M$. In the following we put
\begin{equation}
h_{ij}:=H_{ik\ell}H_{j}{}^{k\ell}.\label{2.2}
\end{equation}

Suppose the metrics
\begin{equation}
d\widehat{s}^{2}_{t}=\frac{1}{2}\widehat{g}_{ij}(x,t)dx^{i}dx^{j}\label{2.3}
\end{equation}
are the solutions of\footnote{In the following computations we don't
need to use the evolution equation for $H(x,t)$, hence we only
consider the evolution equation for metrics.}
\begin{equation}
\frac{\partial}{\partial
t}\widehat{g}_{ij}(x,t)=-2\widehat{R}_{ij}(x,t)+\widehat{h}_{ij}(x,t),
\ \ \ \widehat{g}_{ij}(x,0)=\widetilde{g}_{ij}(x)\label{2.4}
\end{equation}
for a short time $0\leq t\leq T$. Consider a family of smooth
diffeomorphisms $\varphi_{t}: M\to M(0\leq t\leq T)$ of $M$. Let
\begin{equation}
ds^{2}_{t}:=\varphi^{\ast}_{t}d\widehat{s}^{2}_{t}, \ \ \ 0\leq
t\leq T\label{2.5}
\end{equation}
be the pull-back metrics of $d\widehat{s}^{2}_{t}$. For coordinates
system $x=\{x^{1},\cdots,x^{n}\}$ on $M$, let
\begin{equation}
ds^{2}_{t}=g_{ij}(x,t)dx^{i}dx^{j}\label{2.6}
\end{equation}
and
\begin{equation}
y(x,t)=\varphi_{t}(x)=\{y^{1}(x,t),\cdots,y^{n}(x,t)\}.\label{2.7}
\end{equation}
Then we have
\begin{equation}
g_{ij}(x,t)=\frac{\partial y^{\alpha}}{\partial x^{i}}\frac{\partial
y^{\beta}}{\partial x^{j}}\widehat{g}_{\alpha\beta}(y,t).\label{2.8}
\end{equation}
By the assumption $\widehat{g}_{\alpha\beta}(x,t)$ are the solutions
of
\begin{equation}
\frac{\partial}{\partial
t}\widehat{g}_{\alpha\beta}(x,t)=-2\widehat{R}_{\alpha\beta}(x,t)+\widehat{h}_{\alpha\beta}(x,t),
\ \ \ \widehat{g}_{\alpha\beta}(x,0)=\widetilde{g}_{\alpha\beta}(x).\label{2.9}
\end{equation}
We use $R_{ij},\widehat{R}_{ij}, \widetilde{R}_{ij}$;
$\Gamma^{k}_{ij}, \widehat{\Gamma}^{k}_{ij},
\widetilde{\Gamma}^{k}_{ij}$; $\nabla, \widehat{\nabla},
\widetilde{\nabla}$; $h_{ij}, \widehat{h}_{ij}, \widetilde{h}_{ij}$
to denote the Ricci curvatures, Christoffel symbols, covariant
derivatives, and products of the three form $H$ with respect to
$\widetilde{g}_{ij}, \widehat{g}_{ij}, g_{ij}$ respectively. Then
\begin{eqnarray*}
\frac{\partial}{\partial t}g_{ij}(x,t)&=&\frac{\partial
y^{\alpha}}{\partial x^{i}}\frac{\partial y^{\beta}}{\partial
x^{j}}\left(\frac{\partial}{\partial
t}\widehat{g}_{\alpha\beta}(y,t)\right)+\frac{\partial}{\partial
x^{i}}\left(\frac{\partial y^{\alpha}}{\partial
t}\right)\frac{\partial y^{\beta}}{\partial
x^{j}}\widehat{g}_{\alpha\beta}(y,t) \\
&&+ \ \frac{\partial y^{\alpha}}{\partial
x^{i}}\frac{\partial}{\partial x^{j}}\left(\frac{\partial
y^{\beta}}{\partial t}\right)\widehat{g}_{\alpha\beta}(y,t).
\end{eqnarray*}
From (\ref{2.9}) we have
\begin{equation*}
\frac{\partial}{\partial
t}\widehat{g}_{\alpha\beta}(y,t)=-2\widehat{R}_{\alpha\beta}(y,t)
+\widehat{h}_{\alpha\beta}(y,t)+\frac{\partial\widehat{g}_{\alpha\beta}}{\partial
y^{\gamma}}\frac{\partial y^{\gamma}}{\partial t},
\end{equation*}
and
\begin{eqnarray*}
\frac{\partial}{\partial t}g_{ij}(x,t)&=&-2\frac{\partial
y^{\alpha}}{\partial x^{i}}\frac{\partial y^{\beta}}{\partial
x^{j}}\widehat{R}_{\alpha\beta}(y,t)+\frac{\partial
y^{\alpha}}{\partial x^{i}}\frac{\partial y^{\beta}}{\partial
x^{j}}\widehat{h}_{\alpha\beta}(y,t) \\
&&+ \ \frac{\partial y^{\alpha}}{\partial x^{i}}\frac{\partial
y^{\beta}}{\partial
x^{j}}\frac{\partial\widehat{g}_{\alpha\beta}}{\partial
y^{\gamma}}\frac{\partial y^{\gamma}}{\partial
t}+\frac{\partial}{\partial x^{i}}\left(\frac{\partial
y^{\alpha}}{\partial t}\right)\frac{\partial y^{\beta}}{\partial
x^{j}}\widehat{g}_{\alpha\beta}(y,t) \\
&&+ \ \frac{\partial y^{\alpha}}{\partial
x^{i}}\frac{\partial}{\partial x^{j}}\left(\frac{\partial
y^{\beta}}{\partial t}\right)\widehat{g}_{\alpha\beta}(y,t).
\end{eqnarray*}
Since
\begin{equation*}
R_{ij}(x,t)=\frac{\partial y^{\alpha}}{\partial x^{i}}\frac{\partial
y^{\beta}}{\partial x^{j}}\widehat{R}_{\alpha\beta}(y,t), \ \ \
h_{ij}(x,t)=\frac{\partial y^{\alpha}}{\partial x^{i}}\frac{\partial
y^{\beta}}{\partial x^{j}}\widehat{h}_{\alpha\beta}(y,t),
\end{equation*}
using the equation (\cite{Shi}, Sec. 2, (29)), we obtain
\begin{eqnarray}
\frac{\partial}{\partial t}g_{ij}(x,t)&=&-2R_{ij}(x,t)+h_{ij}(x,t)
\nonumber\\
&&+ \ \nabla_{i}\left(\frac{\partial y^{\alpha}}{\partial
t}\frac{\partial x^{k}}{\partial
y^{\alpha}}g_{jk}\right)+\nabla_{j}\left(\frac{\partial
y^{\alpha}}{\partial t}\frac{\partial x^{k}}{\partial
y^{\alpha}}g_{ik}\right).\label{2.10}
\end{eqnarray}
According to DeTurck trick, we define $y(x,t)=\varphi_{t}(x)$ by the
equation
\begin{equation}
\frac{\partial y^{\alpha}}{\partial t}=\frac{\partial
y^{\alpha}}{\partial
x^{k}}g^{\beta\gamma}(\Gamma^{k}_{\beta\gamma}-\widetilde{\Gamma}^{k}_{\beta\gamma}),
\ \ \ y^{\alpha}(x,0)=x^{\alpha},\label{2.11}
\end{equation}
then (\ref{2.10}) becomes
\begin{equation}
\frac{\partial}{\partial t}g_{ij}(x,t)=-2R_{ij}(x,t)+h_{ij}(x,t)
+\nabla_{i}V_{j}+\nabla_{j}V_{i}, \ \ \
g_{ij}(x,0)=\widetilde{g}_{ij}(x),\label{2.12}
\end{equation}
where
\begin{equation}
V_{i}=g_{ik}g^{\beta\gamma}(\Gamma^{k}_{\beta\gamma}
-\widetilde{\Gamma}^{k}_{\beta\gamma}).\label{2.13}
\end{equation}

\begin{lemma} \label{l2.1}The evolution equation (\ref{2.12}) is a strictly parabolic
system. Moreover,
\begin{eqnarray*}
\frac{\partial}{\partial
t}g_{ij}&=&g^{\alpha\beta}\widetilde{\nabla}_{\alpha}\widetilde{\nabla}_{\beta}g_{ij}
-g^{\alpha\beta}g_{ip}\widetilde{g}^{pq}\widetilde{R}_{j\alpha
q\beta}-g^{\alpha\beta}g_{jp}\widetilde{g}^{pq}\widetilde{R}_{i\alpha
q\beta} \\
&&+ \ \frac{1}{2}g^{\alpha\beta}g^{pq}(\widetilde{\nabla}_{i}g_{p\alpha}
\cdot\widetilde{\nabla}_{j}g_{q\beta}+2\widetilde{\nabla}g_{jp}\cdot
\widetilde{\nabla}_{q}g_{i\beta}-2\widetilde{\nabla}_{\alpha}g_{jp}\cdot\widetilde{\nabla}_{\beta}g_{iq}
\\
&&- \ 2\widetilde{\nabla}_{j}g_{p\alpha}\cdot\widetilde{\nabla}_{\beta}g_{iq}
-2\widetilde{\nabla}_{i}g_{p\alpha}\cdot\widetilde{\nabla}_{\beta}g_{jq})+\frac{1}{2}g^{\alpha\beta}g^{pq}H_{i\alpha
p}H_{j\beta q}.
\end{eqnarray*}
\end{lemma}

\begin{proof} It is an immediate consequence of Lemma 2.1 of \cite{Shi}.
\end{proof}

Now we can prove the short-time existence of GRF.

\begin{theorem} \label{t2.2}There is a unique solution to {\rm GRF} for a short time.
More precisely, let $(M,g_{ij}(x))$ be an $n$-dimensional closed
Riemannian manifold with a three form $H=\{H_{ijk}\}$, then there
exists a constant $T=T(n)>0$ depending only on $n$ such that the
evolution system
\begin{eqnarray*}
\frac{\partial}{\partial
t}g_{ij}(x,t)&=&-2R_{ij}(x,t)+\frac{1}{2}g^{kp}(x,t)g^{\ell q}(x,t)H_{ik\ell}(x,t)H_{jpq}(x,t),
\\
\frac{\partial}{\partial t}H(x,t)&=&\Delta_{{\rm HL},g(x,t)}H(x,t), \ \ \
H(x,0)=H(x), \ \ \ g(x,0)=g(x),
\end{eqnarray*}
has a unique solution $(g_{ij}(x,t),H_{ijk}(x,t))$ for a short time
$0\leq t\leq T$.
\end{theorem}
\begin{proof} We proved that the first evolution equation is
strictly parabolic by Lemma \ref{l2.1}. Form the Ricci identity, we have
$\Delta_{{\rm HL},g(x,t)}H=\Delta_{{\rm LB},g(x,t)}H+{\rm Rm}\ast H$ which is also
strictly parabolic. Hence from the standard theory of parabolic
systems, the evolution system has a unique solution.
\end{proof}

\section{Evolution of curvatures}\label{3}

The evolution equation for the Riemann curvature tensors to the
usual Ricci flow (e.g., \cite{CZ, CK, CCGGIIKLLN, H1, MT, Shi}) is given by
\begin{equation}
\frac{\partial}{\partial t}R_{ijk\ell}=\Delta R_{ijk\ell}+\psi_{ijk\ell}
\end{equation}
where
\begin{eqnarray*}
\psi_{ijk\ell}&=&2(B_{ijk\ell}-B_{ij\ell k}-B_{i\ell jk}+B_{ikj\ell}) \\
&&- \ g^{pq}(R_{pjk\ell}R_{q\ell}+R_{ipk\ell}R_{qj}+R_{ijp\ell}
R_{qk}+R_{ijkp}R_{q\ell}),
\end{eqnarray*}
and $B_{ijk\ell}=g^{pr}g^{qs}R_{piqj}R_{rks\ell}$. From this we can easily
deduce the evolution equation for the Riemann curvature tensors to
GRF.
\\

Let $v_{ij}(x,t)$ be any symmetric $2$-tensor, we consider the flow
\begin{equation}
\frac{\partial}{\partial t}g_{ij}(x,t)=v_{ij}(x,t).\label{3.2}
\end{equation}
Applying a formula in \cite{CK} to our case $v_{ij}:=-2R_{ij}+\frac{1}{2}h_{ij}$ with $h_{ij}=H_{ik\ell}H_{j}{}^{k\ell}$, we obtain
\begin{eqnarray*}
\frac{\partial}{\partial
t}R_{ijk\ell}&=&-\frac{1}{2}\left(-2\nabla_{i}\nabla_{k}R_{j\ell}
+\frac{1}{2}\nabla_{i}\nabla_{k}h_{j\ell}+2\nabla_{i}\nabla_{\ell}
R_{jk}-\frac{1}{2}\nabla_{i}\nabla_{\ell}h_{jk}\right.
\\
&&+ \ \left.2\nabla_{j}\nabla_{k}R_{i\ell}-\frac{1}{2}\nabla_{j}\nabla_{k}
h_{i\ell}-2\nabla_{j}\nabla_{\ell}R_{ik}+\frac{1}{2}\nabla_{j}\nabla_{\ell}h_{ik}\right)
\\
&&+ \ \frac{1}{2}g^{pq}\left[R_{ijkp}\left(-2R_{q\ell}
+\frac{1}{2}h_{q\ell}\right)+R_{ijp\ell}\left(-2R_{qk}
+\frac{1}{2}h_{qk}\right)\right]
\\
&=&\nabla_{i}\nabla_{k}R_{j\ell}-\nabla_{i}\nabla_{\ell}R_{jk}
-\nabla_{j}\nabla_{k}R_{i\ell}+\nabla_{j}\nabla_{\ell}R_{ik}
\\
&&- \ g^{pq}(R_{ijkp}R_{q\ell}+R_{ijp\ell}R_{qk}) \\
&&+ \ \frac{1}{4}\left(-\nabla_{i}\nabla_{k}h_{j\ell}
+\nabla_{i}\nabla_{\ell}h_{jk}+\nabla_{j}\nabla_{k}h_{i\ell}
-\nabla_{j}\nabla_{\ell}h_{ik}\right)
\\
&&+ \ \frac{1}{4}g^{pq}\left(R_{ijkp}h_{q\ell}+R_{ijp\ell}h_{qk}\right) \\
&=&\Delta R_{ijk\ell}+2\left(B_{ijk\ell}-B_{ij\ell k}-B_{iljk}+B_{ikj\ell}\right) \\
&&- \ g^{pq}(R_{pjk\ell}R_{q\ell}+R_{ipk\ell}R_{qj}+R_{ijp\ell}R_{qk}+R_{ijkp}R_{q\ell})
\\
&&+ \ \frac{1}{4}\left(-\nabla_{i}\nabla_{k}h_{j\ell}
+\nabla_{i}\nabla_{\ell}h_{jk}+\nabla_{j}\nabla_{k}h_{i\ell}
-\nabla_{j}\nabla_{\ell}h_{ik}\right)
\\
&&+ \ \frac{1}{4}g^{pq}\left(R_{ijkp}h_{q\ell}+R_{ijp\ell}h_{qk}\right).
\end{eqnarray*}

\begin{proposition} \label{p3.1}For {\rm GRF} we have
\begin{eqnarray*}
\frac{\partial}{\partial t}R_{ijk\ell}&=&\Delta R_{ijk\ell}+2\left(B_{ijk\ell}-B_{ij\ell k}-B_{i\ell jk}+B_{ikj\ell}\right) \\
&&- \ g^{pq}(R_{pjk\ell}R_{q\ell}+R_{ipk\ell}R_{qj}+R_{ijp\ell}
R_{qk}+R_{ijkp}R_{q\ell})
\\
&&+ \ \frac{1}{4}\left(-\nabla_{i}\nabla_{k}h_{j\ell}
+\nabla_{i}\nabla_{\ell}h_{jk}+\nabla_{j}\nabla_{k}h_{i\ell}
-\nabla_{j}\nabla_{\ell}h_{ik}\right)
\\
&&+ \ \frac{1}{4}g^{pq}\left(R_{ijkp}h_{q\ell}+R_{ijp\ell}h_{qk}\right).
\end{eqnarray*}
\end{proposition}

In particular,

\begin{corollary} \label{c3.2}For {\rm GRF} we have
\begin{equation}
\frac{\partial}{\partial t}{\rm Rm}=\Delta{\rm Rm}+{\rm Rm}\ast{\rm
Rm}+H\ast H\ast{\rm Rm}+\sum^{2}_{i=0}
\nabla^{i}H\ast\nabla^{2-i}H.\label{3.3}
\end{equation}
\end{corollary}

\begin{proof} From above proposition, we obtain
\begin{equation*}
\frac{\partial}{\partial t}{\rm Rm}=\Delta{\rm Rm}+{\rm Rm}\ast{\rm
Rm}+\nabla^{2}h+h\ast{\rm Rm}.
\end{equation*}
On the other hand, $h=H\ast H$ and
\begin{equation*}
\nabla^{2} h=\nabla(\nabla(H\ast H))=\nabla(\nabla H\ast
H)=\nabla^{2}H\ast H+\nabla H\ast\nabla H.
\end{equation*}
Combining these terms, we obtain the result.
\end{proof}

\begin{proposition} \label{p3.3}For {\rm GRF} we have
\begin{eqnarray*}
\frac{\partial}{\partial t}R_{ik}&=&\Delta R_{ik}+2\langle
R_{piqk},R_{pq}\rangle-2\langle R_{pi},R_{pk}\rangle+\frac{1}{4}[\langle h_{\ell q},R_{i\ell kq}\rangle+\langle
R_{ip},h_{kp}\rangle]\\
&&+ \ \frac{1}{4}\left[-\nabla_{i}\nabla_{k}|H|^{2}+g^{j\ell}\nabla_{i}
\nabla_{\ell}h_{jk}+g^{j\ell}\nabla_{j}\nabla_{k}h_{i\ell}-\Delta
h_{ik}\right].
\end{eqnarray*}
\end{proposition}

\begin{proof} Since
\begin{equation*}
\frac{\partial}{\partial t}R_{ik}=g^{j\ell}\frac{\partial}{\partial
t}R_{ijk\ell}+2g^{jp}g^{\ell q}R_{ijk\ell}R_{pq},
\end{equation*}
and
\begin{equation*}
g^{ij}h_{ij}=g^{ij}H_{ipq}H_{j}{}^{pq}=g^{ij}g^{pr}g^{qs}H_{ipq}H_{jrs}=|H|^{2},
\end{equation*}
it follows that
\begin{eqnarray*}
&
&g^{j\ell}[-\nabla_{i}\nabla_{k}h_{j\ell}+\nabla_{i}\nabla_{\ell}
h_{jk}+\nabla_{j}\nabla_{k}h_{i\ell}
\\
&&- \ \nabla_{j}\nabla_{\ell}h_{ik}+g^{pq}h_{q\ell}R_{ijkp}
+g^{pq}h_{qk}R_{ijp\ell}]
\\
&=&-\nabla_{i}\nabla_{k}|H|^{2}+g^{j\ell}\nabla_{i}\nabla_{\ell}
h_{jk}+g^{j\ell}\nabla_{j}\nabla_{k}h_{i\ell}
\\
&&- \ \Delta h_{ik}+g^{j\ell}g^{pq}h_{q\ell}R_{ijkp}+g^{pq}h_{qk}R_{ip}.
\end{eqnarray*}
From these identities, we get the result.
\end{proof}

As a consequence, we obtain the evolution equation for scalar
curvature.

\begin{proposition}\label{p3.4} For {\rm GRF} we have
\begin{equation*}
\frac{\partial}{\partial t}R=\Delta
R+2|{\rm Ric}|^{2}-\frac{1}{2}\Delta|H|^{2}+\frac{1}{2}\langle
h_{ij},R_{ij}\rangle+\frac{1}{2}g^{ik}g^{j\ell}\nabla_{i}\nabla_{j}h_{k\ell}.
\end{equation*}
\end{proposition}

\begin{proof} From the usual evolution equation for scalar curvature
under the Ricci flow, we have
\begin{eqnarray*}
\frac{\partial}{\partial t}R&=&\Delta
R+2|{\rm Ric}|^{2}+\frac{1}{4}g^{ik}[\langle h_{\ell q},R_{i\ell kq}\rangle+\langle
R_{ip},h_{kp}\rangle] \\
&&+ \ \frac{1}{4}g^{ik}\left(-\nabla_{i}\nabla_{k}|H|^{2}+g^{j\ell}
\nabla_{i}\nabla_{\ell}h_{jk}+g^{j\ell}\nabla_{j}\nabla_{k}h_{i\ell}-\Delta
h_{ik}\right) \\
&=&\Delta R+2|{\rm Ric}|^{2}+\frac{1}{4}\langle
h_{ij},R_{ij}\rangle+\frac{1}{4}\langle R_{ip},h_{ip}\rangle \\
&&- \ \frac{1}{4}\Delta|H|^{2}+\frac{1}{4}g^{ik}g^{j\ell}\nabla_{i}
\nabla_{\ell}h_{jk}+\frac{1}{4}g^{ik}g^{j\ell}\nabla_{j}\nabla_{k}h_{i\ell}
-\frac{1}{4}\Delta|H|^{2}.
\end{eqnarray*}
Simplifying the terms, we obtain the required result.
\end{proof}

\section{Derivative estimates}\label{4}

In this section we are going to prove BBS estimates. At first we
review several basic identities of commutators $[\Delta, \nabla]$
and $[\frac{\partial}{\partial t}, \nabla]$. If $A=A(t)$ is a
$t$-dependency tensor, and $\frac{\partial g_{ij}}{\partial
t}=v_{ij}$, then applying the well-known formulas stated in \cite{CK} on {\rm GRF} we have
\begin{eqnarray}
\frac{\partial}{\partial t}\nabla{\rm
Rm}&=&\nabla\frac{\partial}{\partial t}{\rm Rm}+{\rm
Rm}\ast\nabla({\rm Rm}+H\ast H) \nonumber\\
&=&\nabla(\Delta{\rm Rm}+{\rm Rm}\ast{\rm Rm}+H\ast H\ast{\rm
Rm}+\nabla^{2}H\ast H+\nabla H\ast\nabla H) \nonumber\\
&&+ \ {\rm Rm}\ast\nabla{\rm Rm}+H\ast\nabla H\ast{\rm Rm}\label{4.1} \\
&=&\Delta(\nabla{\rm Rm})+\sum_{i+j=1}\nabla^{i}{\rm
Rm}\ast\nabla^{j}{\rm Rm} \nonumber\\
&&+ \ \sum_{i+j+k=1}\nabla^{i}H\ast\nabla^{j}H\ast\nabla^{k}{\rm Rm}+\sum_{i+j=1+2}\nabla^{i}H\ast\nabla^{j}H.\nonumber
\end{eqnarray}

More generally, we have

\begin{proposition} \label{p4.1}For {\rm GRF} and any nonnegative
integer $\ell$ we have
\begin{eqnarray}
\frac{\partial}{\partial t}\nabla^{\ell}{\rm
Rm}&=&\Delta(\nabla^{\ell}{\rm Rm})+\sum_{i+j=\ell}\nabla^{i}{\rm
Rm}\ast\nabla^{j}{\rm Rm}
\nonumber\\
&&+ \ \sum_{i+j+k=\ell}\nabla^{i}H\ast\nabla^{j}H\ast\nabla^{k}{\rm Rm}+\sum_{i+j=\ell+2}\nabla^{i}H\ast\nabla^{j}H.\label{4.2}
\end{eqnarray}
\end{proposition}

\begin{proof} For $\ell=1$, we have proved before the proposition.
Suppose that the formula holds for $1,\cdots,\ell$. By induction to
$\ell$, for $\ell+1$ we have

\begin{eqnarray*}
\frac{\partial}{\partial t}\nabla^{\ell+1}{\rm
Rm}&=&\frac{\partial}{\partial t}\nabla(\nabla^{\ell}{\rm Rm})\\
&=&\nabla\frac{\partial}{\partial t}(\nabla^{\ell}{\rm
Rm})+\nabla^{\ell}{\rm
Rm}\ast\nabla({\rm Rm}+H\ast H) \\
&=&\nabla\left[\Delta(\nabla^{\ell}{\rm Rm})+\sum_{i+j=\ell}\nabla^{i}{\rm
Rm}\ast\nabla^{j}{\rm Rm}\right. \\
&&+ \ \left.\sum_{i+j+k=\ell}\nabla^{i}H\ast\nabla^{j}H\ast\nabla^{k}{\rm
Rm}+\sum_{i+j=\ell+2}\nabla^{i}H\ast\nabla^{j}H\right] \\
&&+ \ \nabla^{\ell}{\rm Rm}\ast\nabla{\rm Rm}+H\ast\nabla
H\ast\nabla^{\ell}{\rm Rm} \\
&=&\Delta(\nabla^{\ell+1}{\rm Rm})+\nabla{\rm Rm}\ast\nabla^{\ell}{\rm
Rm}+{\rm Rm}\ast\nabla^{\ell+1}{\rm Rm} \\
&&+ \ \sum_{i+j=\ell}\left(\nabla^{i+1}{\rm Rm}\ast\nabla^{j}{\rm
Rm}+\nabla^{i}{\rm Rm}\ast\nabla^{j+1}{\rm Rm}\right) \\
&&+ \ \sum_{i+j+k=\ell}\left(\nabla^{i+1}H\ast\nabla^{j}H\ast\nabla^{k}{\rm
Rm}\right. \\
&&+ \ \left.\nabla^{i}H\ast\nabla^{j+1}H\ast\nabla^{k}{\rm
Rm}+\nabla^{i}H\ast\nabla^{j}H\ast\nabla^{k+1}{\rm Rm}\right) \\
&&+ \ \sum_{i+j=\ell+2}(\nabla^{i+1}H\ast\nabla^{j}H+\nabla^{i}H\ast\nabla^{j+1}H)
\\
&&+ \ H\ast\nabla H\ast\nabla^{l}{\rm Rm}.
\end{eqnarray*}
Simplifying these terms, we obtain the required result.
\end{proof}

As an immediate consequence, we have an evolution inequality for
$|\nabla^{l}{\rm Rm}|^{2}$.

\begin{corollary} \label{c4.2}For {\rm GRF} and any nonnegative integer $\ell$ we have
\begin{eqnarray}
\frac{\partial}{\partial t}|\nabla^{\ell}{\rm
Rm}|^{2}&\leq&\Delta|\nabla^{l}{\rm Rm}|^{2} \ - \
2|\nabla^{\ell+1}{\rm Rm}|^{2}\nonumber\\
&&+ \ C\cdot\sum_{i+j=\ell}|\nabla^{i}{\rm Rm}|\cdot|\nabla^{j}{\rm
Rm}|\cdot|\nabla^{\ell}{\rm Rm}|\nonumber \\
&&+ \ C\cdot\sum_{i+j+k=\ell}|\nabla^{i}H|\cdot|\nabla^{j}H|\cdot|\nabla^{k}{\rm
Rm}|\cdot|\nabla^{\ell}{\rm Rm}|\nonumber \\
&&+ \ C\cdot\sum_{i+j=\ell+2}|\nabla^{i}H|\cdot|\nabla^{j}H
|\cdot|\nabla^{\ell}{\rm
Rm}|,\label{4.3}
\end{eqnarray}
where $C$ are universal constants depending only on the dimension of
$M$.
\end{corollary}

Next we derive the evolution equations for the covariant derivatives
of $H$.

\begin{proposition} \label{p4.3}For {\rm GRF} and any positive integer $\ell$ we have
\begin{equation}
\frac{\partial}{\partial
t}\nabla^{\ell}H=\Delta(\nabla^{\ell}H)+\sum_{i+j=\ell}
\nabla^{i}H\ast\nabla^{j}{\rm
Rm}+\sum_{i+j+k=\ell}\nabla^{i}H\ast\nabla^{j}H\ast\nabla^{k}H.\label{4.4}
\end{equation}
\end{proposition}

\begin{proof} From the Bochner formula, the evolution equation for
$H$ can be rewritten as
\begin{equation}
\frac{\partial}{\partial t}H=\Delta H+{\rm Rm}\ast H.\label{4.5}
\end{equation}
For $\ell=1$, we have
\begin{eqnarray*}
\frac{\partial}{\partial t}\nabla H&=&\nabla\frac{\partial}{\partial
t}H+H\ast\nabla({\rm Rm}+H\ast H) \\
&=&\nabla(\Delta H+{\rm Rm}\ast H)+H\ast\nabla{\rm Rm}+H\ast
H\ast\nabla H \\
&=&\nabla(\Delta H)+H\ast\nabla{\rm Rm}+\nabla H\ast{\rm Rm}+H\ast
H\ast\nabla H \\
&=&\Delta(\nabla H)+\nabla{\rm Rm}\ast H+\nabla H\ast{\rm Rm}+H\ast
H\ast\nabla H.
\end{eqnarray*}
Using (\ref{4.2}) and the same argument, we can prove the evolution
equation for higher covariant derivatives.
\end{proof}

Similarly, we have an evolution inequality for $|\nabla^{\ell}H|^{2}$.

\begin{corollary} \label{c4.4}For {\rm GRF} and for any positive integer $l$ we have
\begin{eqnarray}
\frac{\partial}{\partial
t}|\nabla^{l}H|^{2}&\leq&\Delta|\nabla^{\ell}H|^{2}-2|\nabla^{\ell+1}H|^{2}
\nonumber \\
&&+ \ C\cdot\sum_{i+j=\ell}|\nabla^{i}H|\cdot|\nabla^{j}{\rm
Rm}|\cdot|\nabla^{\ell}H|\label{4.6}\\
&&+ \ C\cdot\sum_{i+j+k=\ell}|\nabla^{i}H|\cdot|\nabla^{j}
H|\cdot|\nabla^{k}H|\cdot|\nabla^{l}H|,\nonumber
\end{eqnarray}
while
\begin{equation}
\frac{\partial}{\partial t}|H|^{2}\leq\Delta|H|^{2}-2|\nabla
H|^{2}+C\cdot|{\rm Rm}|\cdot|H|^{2}.\label{4.7}
\end{equation}
\end{corollary}

\begin{theorem} \label{t4.5}Suppose that $(g(x,t),H(x,t))$ is a solution to {\rm GRF} on a
closed manifold $M^{n}$ for a short time $0\leq t\leq T$ and
$K_{1}, K_{2}$ are arbitrary given nonnegative constants. Then there
exists a constant $C_{n}$ depending only on $n$ such that if
\begin{equation*}
|{\rm Rm}(x,t)|_{g(x,t)}\leq K_{1}, \ \ \ |H(x)|_{g(x)}\leq K_{2}
\end{equation*}
for all $x\in M$ and $t\in[0,T]$, then
\begin{equation}
|H(x,t)|_{g(x,t)}\leq K_{2}e^{C_{n}K_{1}t}\label{4.8}
\end{equation}
for all $x\in M$ and $t\in[0,T]$.
\end{theorem}

\begin{proof} Since
\begin{equation*}
\frac{\partial}{\partial t}|H|^{2}\leq\Delta|H|^{2}+C_{n}|{\rm
Rm}|\cdot|H|^{2}\leq \Delta|H|^{2}+C_{n}K_{1}|H|^{2},
\end{equation*}
using maximum principle, we obtain $u(t)\leq u(0)e^{C_{n}K_{1}t}$,
where $u(t)=|H|^{2}$.
\end{proof}

The main result in this section is the following estimates for
higher derivatives of Riemann curvature tensors and three forms. Some
special cases were proved in \cite{S1,S2,Y}.

\begin{theorem} \label{t4.6}Suppose that $(g(x,t),H(x,t))$ is a solution to {\rm GRF} on a
compact manifold $M^{n}$ and $K$ is an arbitrary given positive
constant. Then for each $\alpha>0$ and each integer $m\geq1$ there
exists a constant $C_{m}$ depending on $m, n, \max\{\alpha,1\}$, and
$K$ such that if
\begin{equation*}
|{\rm Rm}(x,t)|_{g(x,t)}\leq K, \ \ \ |H(x)|_{g(x)}\leq K
\end{equation*}
for all $x\in M$ and $t\in[0,\alpha/K]$, then
\begin{equation}
|\nabla^{m-1} {\rm Rm}(x,t)|_{g(x,t)}+|\nabla^{m}
H(x,t)|_{g(x,t)}\leq\frac{C_{m}}{t^{m/2}}\label{4.9}
\end{equation}
for all $x\in M$ and $t\in(0,\alpha/K]$.
\end{theorem}

\begin{proof} In the following computations we always let $C$ be any constants depending on
$n,m,\max\{\alpha,1\}$,and $K$, which may take different values at
different places. From the evolution equations and Theorem \ref{t4.5}, we
have
\begin{eqnarray*}
\frac{\partial}{\partial t}|{\rm Rm}|^{2}&\leq&\Delta|{\rm
Rm}|^{2}-2|\nabla{\rm Rm}|^{2}+C+C|\nabla^{2} H|+C|\nabla H|^{2}, \\
\frac{\partial}{\partial t}|H|^{2}&\leq&\Delta|H|^{2}-2|\nabla
H|^{2}+C, \\
\frac{\partial}{\partial t}|\nabla H|^{2}&\leq&\Delta|\nabla
H|^{2}-2|\nabla^{2} H|^{2}+C|\nabla{\rm Rm}|\cdot|\nabla H|+C|\nabla
H|^{2}.
\end{eqnarray*}
Consider the function $u=t|\nabla H|^{2}+\gamma|H|^{2}+t|{\rm
Rm}|^{2}$. Directly computing, we obtain
\begin{eqnarray*}
\frac{\partial}{\partial t}u&\leq&\Delta
u-2t|\nabla^{2}H|^{2}+Ct|\nabla^{2}H|+(C-2\gamma)|\nabla
H|^{2}+C+C\gamma \\
&&- \ 2t|\nabla{\rm Rm}|^{2}+Ct\cdot|\nabla{\rm Rm}|\cdot|\nabla H| \\
&\leq&\Delta u+2(C-\gamma)\cdot|\nabla H|^{2}+C(1+\gamma).
\end{eqnarray*}
If we choose $\gamma=C$, then $\frac{\partial}{\partial
t}u\leq\Delta u+C$ which implies that $u\leq Ce^{Ct}$ since $u(0)\leq C$.
With this estimate we are able to bound the first covariant
derivative of ${\rm Rm}$ and the second covariant derivative of $H$.
In order to control the term $|\nabla{\rm Rm}|^{2}$, we should use
the evolution equations of $|H|^{2}$, $|\nabla H|^{2}$ and
$|\nabla^{2}H|^{2}$ to cancel with the bad terms, i.e.,
$|\nabla^{2}{\rm Rm}|^{2}$, $|\nabla^{2}H|^{2}$, and
$|\nabla^{3}H|^{2}$,  in the evolution equation of $|\nabla{\rm
Rm}|^{2}$:
\begin{eqnarray*}
\frac{\partial}{\partial t}|\nabla{\rm
Rm}|^{2}&\leq&\Delta|\nabla{\rm Rm}|^{2}-2|\nabla^{2}{\rm
Rm}|^{2}+C|\nabla{\rm Rm}|^{2} +\frac{C}{t^{1/2}}|\nabla{\rm Rm}|\\
&&+ \ C\cdot|\nabla{\rm
Rm}|\cdot|\nabla^{3}H|+\frac{C}{t^{1/2}}|\nabla^{2}H|\cdot|\nabla{\rm
Rm}|, \\
\frac{\partial}{\partial t}|\nabla^{2}H|^{2}&\leq&\Delta|\nabla^{2}
H|^{2}-2|\nabla^{3}H|^{2}+C\cdot|\nabla^{2}{\rm
Rm}|\cdot|\nabla^{2}H| \\
&&+ \ \frac{C}{t^{1/2}}|\nabla{\rm
Rm}|\cdot|\nabla^{2}H|+C|\nabla^{2}H|^{2}+\frac{C}{t}|\nabla^{2}H|.
\end{eqnarray*}
As above, we define
\begin{equation*}
u:=t^{2}(|\nabla^{2}H|^{2}+|\nabla{\rm Rm}|^{2})+t\beta(|\nabla
H|^{2}+|{\rm Rm}|^{2})+\gamma|H|^{2},
\end{equation*}
and therefore, $\frac{\partial u}{\partial t}\leq\Delta u+ C$.
Motivated by cases for $m=1$ and $m=2$, for general $m$, we can
define a function
\begin{equation*}
u:=t^{m}(|\nabla^{m}H|^{2}+|\nabla^{m-1}{\rm
Rm}|^{2})+\sum^{m-1}_{i=1}\beta_{i}t^{i}(|\nabla^{i}H|^{2}+|\nabla^{i-1}{\rm
Rm}|^{2})+\gamma|H|^{2},
\end{equation*}
where $\beta_{i}$ and $\gamma$ are positive constants determined
later. In the following, we always assume $m\geq3$. Suppose
\begin{equation*}
|\nabla^{i-1}{\rm Rm}|+|\nabla^{i} H|\leq\frac{C_{i}}{t^{i/2}}, \ \
\ i=1,2,\cdots,m-1.
\end{equation*}
For such $i$, from Corollary \ref{c4.4}, we have
\begin{eqnarray*}
\frac{\partial}{\partial
t}|\nabla^{i}H|^{2}&\leq&\Delta|\nabla^{i}H|^{2}-2|\nabla^{i+1}H|^{2}
+C\sum^{i}_{j=0}|\nabla^{j}H|\cdot|\nabla^{i-j}{\rm
Rm}|\cdot|\nabla^{i}H|\\
&&+ \ C\sum^{i}_{j=0}\sum^{i-j}_{\ell=0}|\nabla^{j}H|\cdot|\nabla^{i-j
-\ell}H|\cdot|\nabla^{\ell}H|\cdot|\nabla^{i}H|
\\
&\leq&\Delta|\nabla^{i}H|^{2}-2|\nabla^{i+1}H|^{2}+C\cdot|\nabla^{i}H|\sum^{i}_{j=0}\frac{C_{j}}{t^{\frac{j}{2}}}\cdot\frac{C_{i-j+1}}{t^{\frac{i-j+1}{2}}}
\\
&&+ \ C\cdot|\nabla^{i}H|\sum^{i}_{j=0}\sum^{i-j}_{\ell=0}
\frac{C_{j}}{t^{\frac{j}{2}}}\cdot\frac{C_{i-j-\ell}}
{t^{\frac{i-j-1}{2}}}\cdot\frac{C_{\ell}}{t^{\frac{l}{2}}}
\\
&\leq&\Delta|\nabla^{i}H|^{2}-2|\nabla^{i+1}H|^{2}+\frac{C_{i}}{t^{\frac{i+1}{2}}}|\nabla^{i}H|+\frac{C_{i}}{t^{\frac{i}{2}}}|\nabla^{i}H|.
\end{eqnarray*}
Similarly, from Corollary \ref{c4.2} we also have
\begin{eqnarray*}
\frac{\partial}{\partial t}|\nabla^{i-1}{\rm
Rm}|^{2}&\leq&\Delta|\nabla^{i-1}{\rm Rm}|^{2}-2|\nabla^{i}{\rm
Rm}|^{2}+C\sum^{i-1}_{j=0}|\nabla^{j}{\rm Rm}||\nabla^{i-1-j}{\rm
Rm}||\nabla^{i-1}{\rm Rm}| \\
&&+ \ C\sum^{i-1}_{j=0}\sum^{i-1-j}_{\ell=0}|\nabla^{j}H|
\cdot|\nabla^{i-1-j-\ell}H|\cdot|\nabla^{\ell}{\rm
Rm}|\cdot|\nabla^{i-1}{\rm Rm}| \\
&&+ \ C\sum^{i+1}_{j=0}|\nabla^{j}H|\cdot|\nabla^{i+1-j}H|\cdot|\nabla^{i-1}{\rm
Rm}| \\
&\leq&\Delta|\nabla^{i-1}{\rm Rm}|^{2}-2|\nabla^{i}{\rm Rm}|^{2}
+C\cdot|\nabla^{i-1}{\rm
Rm}|\sum^{i-1}_{j=0}\frac{C_{j+1}}{t^{\frac{j+1}{2}}}\cdot
\frac{C_{i-j}}{t^{\frac{i-j}{2}}}
\\
&&+ \ C\cdot|\nabla^{i-1}{\rm
Rm}|\sum^{i-1}_{j=0}\sum^{i-1-j}_{\ell=0}\frac{C_{j}}{t^{\frac{j}{2}}}
\cdot\frac{C_{i-1-j-\ell}}{t^{\frac{i-1-j-\ell}{2}}}
\cdot\frac{C_{\ell+1}}{t^{\frac{\ell+1}{2}}}
\\
&&+ \ C\cdot|\nabla^{i-1}{\rm
Rm}|\sum^{i}_{j=1}\frac{C_{j}}{t^{\frac{j}{2}}}\cdot\frac{C_{i+1-j}}{t^{\frac{i+1-j}{2}}}+C\cdot|\nabla^{i+1}H|\cdot\frac{C_{i}}{t^{\frac{i}{2}}}
\\
&\leq&\Delta|\nabla^{i-1}{\rm Rm}|^{2}-2|\nabla^{i}{\rm Rm}|^{2}
\\
&&+ \ \frac{C_{i}}{t^{\frac{i+1}{2}}}\cdot|\nabla^{i-1}{\rm
Rm}|+\frac{C_{i}}{t^{\frac{i}{2}}}|\nabla^{i+1}H|+\frac{C_{i}}{t^{\frac{i}{2}}}|\nabla^{i-1}{\rm
Rm}|.
\end{eqnarray*}
The evolution inequality for $u$ now is given by
\begin{eqnarray*}
\frac{\partial u}{\partial
t}&\leq&mt^{m-1}(|\nabla^{m}H|^{2}+|\nabla^{m-1}{\rm Rm}|^{2})
+\sum^{m-1}_{i=1}i\beta_{i}t^{i-1}(|\nabla^{i}H|^{2}+|\nabla^{i-1}{\rm
Rm}|^{2})\\
&&+ \ t^{m}\left(\frac{\partial}{\partial
t}|\nabla^{m}H|^{2}+\frac{\partial}{\partial t}|\nabla^{m-1}{\rm
Rm}|^{2}\right)
\\
&&+ \ \sum^{m-1}_{i=1}\beta_{i}t^{i}\left(\frac{\partial}{\partial
t}|\nabla^{i}H|^{2}+\frac{\partial}{\partial t}|\nabla^{i-1}{\rm
Rm}|^{2}\right)+\gamma\cdot\frac{\partial}{\partial t}|H|^{2}.
\end{eqnarray*}
It's easy to see that the second term is bounded by
\begin{equation*}
\sum^{m-1}_{i=1}i\beta_{i}t^{i-1}\frac{C_{i}}{t^{i}}=\sum^{m-1}_{i=1}i\beta_{i}C_{i}t^{-1},
\end{equation*}
but this bound depends on $t$ and approaches to infinity when $t$
goes to zero. Hence we use the last second term to control this bad
term. The evolution inequality for the third term is the combination
of the following two inequalities
\begin{eqnarray*}
\frac{\partial}{\partial
t}|\nabla^{m}H|^{2}&\leq&\Delta|\nabla^{m}H|^{2}-2|\nabla^{m+1}H|^{2}
+C\sum^{m}_{i=0}|\nabla^{i}H|\cdot|\nabla^{m-i}{\rm
Rm}|\cdot|\nabla^{m}H| \\
&&+ \ C\sum^{m}_{i=0}\sum^{m-i}_{j=0}|\nabla^{j}H|\cdot|\nabla^{m-i-j}H|\cdot|\nabla^{i}H|\cdot|\nabla^{m}H|
\\
&\leq&\Delta|\nabla^{m}H|^{2}-2|\nabla^{m+1}H|^{2}+C|\nabla^{m}H|^{2}
\\
&&+ \ C\cdot|\nabla^{m}{\rm Rm}|\cdot|\nabla^{m}H|+\frac{C_{m}}{t^{\frac{m+1}{2}}}|\nabla^{m}H|+\frac{C_{m}}{t^{\frac{m}{2}}}|\nabla^{m}H|,
\end{eqnarray*}
and
\begin{eqnarray*}
\frac{\partial}{\partial t}|\nabla^{m-1}{\rm
Rm}|^{2}&\leq&\Delta|\nabla^{m-1}{\rm Rm}|^{2}-2|\nabla^{m}{\rm
Rm}|^{2} \\
&&+ \ C\sum^{m-1}_{i=0}|\nabla^{i}{\rm Rm}|\cdot|\nabla^{m-1-i}{\rm
Rm}|\cdot|\nabla^{m-1}{\rm Rm}| \\
&&+ \ C\sum^{m-1}_{i=0}\sum^{m-1-i}_{j=0}|\nabla^{j}H|\cdot|\nabla^{m-1-i-j}H|
\cdot|\nabla^{i}{\rm
Rm}|\cdot|\nabla^{m-1}{\rm Rm}| \\
&&+ \ C\sum^{m+1}_{i=0}|\nabla^{i}H|\cdot|\nabla^{m+1-i}H|\cdot|\nabla^{m-1}{\rm
Rm}| \\
&\leq&\Delta|\nabla^{m-1}{\rm Rm}|^{2}-2|\nabla^{m}{\rm Rm}|^{2}
\\
&&+ \ C|\nabla^{m-1}{\rm
Rm}|^{2}+\frac{C}{t^{\frac{1}{2}}}\cdot|\nabla^{m}H|\cdot|\nabla^{m-1}{\rm
Rm}| \\
&&+ \ C|\nabla^{m+1}H|
|\nabla^{m-1}{\rm Rm}|+\frac{C_{m}}{t^{\frac{m+1}{2}}}|\nabla^{m-1}{\rm
Rm}|+\frac{C_{m}}{t^{\frac{m}{2}}}|\nabla^{m-1}{\rm Rm}|.
\end{eqnarray*}

Therefore we have
\begin{eqnarray*}
\frac{\partial u}{\partial
t}&\leq&mt^{m-1}(|\nabla^{m}H|^{2}+|\nabla^{m-1}{\rm Rm}|^{2})
\\
&&+ \ \sum^{m-1}_{i=1}i\beta_{i}t^{i-1}(|\nabla^{i}H|^{2}+|\nabla^{i-1}{\rm
Rm}|^{2}) \\
&&+ \ t^{m}\left(\Delta|\nabla^{m}H|^{2}-2|\nabla^{m+1}H|^{2}+\frac{C}{t^{\frac{m+1}{2}}}|\nabla^{m}H|+C|\nabla^{m}H|^{2}\right.
\\
&&+ \ \left.C|\nabla^{m}{\rm
Rm}|\cdot|\nabla^{m}H|+\Delta|\nabla^{m-1}{\rm Rm}|^{2}\right.
\\
&&- \ \left.2|\nabla^{m}{\rm
Rm}|^{2}+\frac{C}{t^{\frac{m+1}{2}}}|\nabla^{m-1}{\rm
Rm}|+C|\nabla^{m-1}{\rm Rm}|^{2}\right.
\\
&&+ \ \left.\frac{C}{t^{1/2}}|\nabla^{m}H|\cdot|\nabla^{m-1}{\rm
Rm}|+C|\nabla^{m+1}H|\cdot|\nabla^{m-1}{\rm Rm}|\right) \\
&&+ \ \sum^{m-1}_{i=1}\beta_{i}t^{i}\left(\frac{C_{i}}{t^{\frac{i+1}{2}}}|\nabla^{i-1}{\rm
Rm}|+\Delta|\nabla^{i}H|^{2}-2|\nabla^{i+1}H|^{2}\right. \\
&&+ \ \left.\Delta|\nabla^{i-1}{\rm
Rm}|^{2}+\frac{C_{i}}{t^{\frac{i+1}{2}}}|\nabla^{i}H|+\frac{C_{i}}{t^{\frac{i}{2}}}|\nabla^{i+1}H|-2|\nabla^{i}{\rm
Rm}|^{2}\right) \\
&&+ \ \gamma(\Delta|H|^{2}-2|\nabla H|^{2}+C)
\\
&\leq&\Delta
u-2t^{m}|\nabla^{m+1}H|^{2}+Ct^{m}|\nabla^{m+1}H|\cdot|\nabla^{m-1}{\rm
Rm}| \\
&&- \ 2t^{m}|\nabla^{m}{\rm Rm}|^{2}+Ct^{m}|\nabla^{m}{\rm
Rm}|\cdot|\nabla^{m} H| \\
&&+ \ \sum^{m-2}_{i=0}(i+1)\beta_{i+1}t^{i}(|\nabla^{i+1}H|^{2}+|\nabla^{i}{\rm
Rm}|^{2})
\\
&&- \ 2\sum^{m-1}_{i=1}\beta_{i}t^{i}(|\nabla^{i+1}H|^{2}+|\nabla^{i}{\rm
Rm}|^{2})-2\gamma|\nabla H|^{2}+\gamma C \\
&&+ \ Ct^{m-1}|\nabla^{m}H|^{2}+Ct^{m-1}|\nabla^{m-1}{\rm Rm}|^{2}
\\
&&+ \ Ct^{\frac{m-1}{2}}|\nabla^{m}H|+Ct^{\frac{m-1}{2}}|\nabla^{m-1}{\rm
Rm}|\\
&&+ \ Ct^{m-\frac{1}{2}}|\nabla^{m}H|\cdot|\nabla^{m-1}{\rm
Rm}|+Ct^{m}|\nabla^{m+1}H|\cdot|\nabla^{m-1}{\rm Rm}| \\
&&+ \ \sum^{m-1}_{i=1}\beta_{i}C_{i}t^{\frac{i}{2}}|\nabla^{i+1}H|+\sum^{m-1}_{i=1}\beta_{i}C_{i}t^{\frac{i-1}{2}}(|\nabla^{i}H|^{2}+|\nabla^{i-1}{\rm
Rm}|).
\end{eqnarray*}
Choosing
\begin{equation*}
(i+1)\beta_{i+1}=\beta_{i}, \ \ \ \beta_{i}=\frac{A}{i!}, \ \ \
i\geq0,
\end{equation*}
where $A$ is constant which is determined later, and noting that
\begin{eqnarray*}
\sum^{m-1}_{i=1}\beta_{i}C_{i}t^{i/2}|\nabla^{i+1}H|&\leq&\frac{1}{2}\sum^{m-1}_{i=1}\beta_{i}t^{i}|\nabla^{i+1}H|^{2}+\frac{1}{2}\sum^{m-1}_{i=1}\beta_{i}C^{2}_{i},
\end{eqnarray*}
\begin{eqnarray*}
& &
\sum^{m-1}_{i=1}\beta_{i}C_{i}t^{\frac{i-1}{2}}(|\nabla^{i}H|+|\nabla^{i-1}{\rm
Rm}|) \\
&\leq&\beta_{1}C_{1}(|\nabla
H|+|{\rm Rm}|)+\sum^{m-2}_{i=1}\beta_{i+1}C_{i+1}t^{\frac{i}{2}}(|\nabla^{i+1}H|+|\nabla^{i}{\rm Rm}|) \\
&\leq&\beta_{1}C_{1}(|\nabla H|+|{\rm Rm}|)
\\
&&+ \ \sum^{m-2}_{i=1}\beta_{i+1}C_{i+1}\left(\frac{t^{i}|\nabla^{i+1}H|^{2}}{\frac{2\beta_{i+1}C_{i+1}}{\beta_{i}}}+\frac{t^{i}|\nabla^{i}{\rm
Rm}|^{2}}{\frac{2\beta_{i+1}C_{i+1}}{\beta_{i}}}+\frac{\beta_{i+1}C_{i+1}}{\beta_{i}}\right)
\\
&\leq&\beta_{1}C_{1}(|\nabla H|+|{\rm Rm}|)
\\
&&+ \ \frac{1}{2}\sum^{m-2}_{i=1}\beta_{i}t^{i}(|\nabla^{i+1}H|^{2}+|\nabla^{i}{\rm
Rm}|^{2})+\sum^{m-2}_{i=1}\frac{\beta^{2}_{i+1}C^{2}_{i+1}}{\beta_{i}},
\end{eqnarray*}
It yields that
\begin{eqnarray*}
\frac{\partial}{\partial t}u&\leq&\Delta u-2t^{m}|\nabla^{m+1}H|^{2}+Ct^{m}|\nabla^{m+1}H|\cdot|\nabla^{m-1}{\rm Rm}| \\
&&- \ 2t^{m}|\nabla^{m}{\rm Rm}|^{2}+Ct^{m}|\nabla^{m}H|\cdot|\nabla^{m}{\rm Rm}| \\
&&+ \ Ct^{m-1}|\nabla^{m}H|^{2}+Ct^{m-1}|\nabla^{m-1}{\rm Rm}|^{2} \\
&&+ \ Ct^{m-\frac{1}{2}}|\nabla^{m}H|\cdot|\nabla^{m-1}{\rm
Rm}|+\beta_{0}(|\nabla H|^{2}+|{\rm Rm}|^{2})
\\
&&- \ \sum^{m-1}_{i=1}\beta_{i}t^{i}(|\nabla^{i+1}H|^{2}+|\nabla^{i}{\rm
Rm}|^{2})\\
&&+ \ \sum^{m-2}_{i=1}\beta_{i}t^{i}(|\nabla^{i+1}H|^{2}+|\nabla^{i}{\rm
Rm}|^{2})+\frac{1}{2}\beta_{m-1}t^{m-1}|\nabla^{m}H|^{2} \\
&&+ \ \beta_{1}C_{1}|\nabla H|-2\gamma|\nabla H|^{2}+C+C\gamma \\
&\leq&\Delta u+Ct^{m-1}|\nabla^{m-1}{\rm
Rm}|^{2}+Ct^{m-1}|\nabla^{m}H|^{2} \\
&&+ \ Ct^{m-\frac{1}{2}}(|\nabla^{m}H|^{2}+|\nabla^{m-1}{\rm
Rm}|^{2})+\beta_{0}|\nabla H|^{2} \\
&&+ \ \beta_{1}C_{1}|\nabla H|-2\gamma|\nabla H|^{2}+C+C\gamma \\
&&- \ \frac{1}{2}\beta_{m-1}t^{m-1}|\nabla^{m}H|^{2}-\beta_{m-1}t^{m-1}|\nabla^{m}{\rm
Rm}|^{2} \\
&\leq&\Delta
u+\frac{1}{2}(C\sqrt{t}+C-\beta_{m-1})t^{m-1}(|\nabla^{m-1}{\rm
Rm}|^{2}+|\nabla^{m}H|^{2}) \\
&&+ \ (\beta_{0}+\beta_{1}C_{1}-2\gamma)|\nabla
H|^{2}+C+C\gamma+\beta_{1}C_{1}.
\end{eqnarray*}
When we chose $A$ and $\gamma$ sufficiently large, we obtain
$\frac{\partial u}{\partial t}\leq\Delta u+C$ which implies that
$u(t)\leq C$ since $u(0)$ is bounded.
\end{proof}

Finally we give an estimate which plays a crucial role in the next
section.

\begin{corollary} \label{c4.7}Let $(g(x,t),H(x,t))$ be a solution of the
generalized Ricci flow on a closed manifold $M$. If there are
$\beta>0$ and $K>0$ such that
\begin{equation*}
|{\rm Rm}(x,t)|_{g(x,t)}\leq K, \ \ \ |H(x)|_{g(x)}\leq K
\end{equation*}
for all $x\in M$ and $t\in[0,T]$, where $T>\beta/K$, then
there exists for each $m\in\mathbb{N}$ a constant $C_{m}$ depending on $m,n,
\min\{\beta,1\}$, and $K$ such that
\begin{equation*}
|\nabla^{m-1}{\rm Rm}(x,t)|_{g(x,t)}+|\nabla^{m}H(x,t)|_{g(x,t)}\leq
C_{m}K^{m/2}
\end{equation*}
for all $x\in M$ and $t\in[\min\{\beta,1\}/K,T]$.
\end{corollary}

\begin{proof} The proof is the same as in \cite{CCGGIIKLLN}, we just copy it here. Let
$\beta_{1}:=\min\{\beta,1\}$. For any fixed point
$t_{0}\in[\beta_{1}/K,T]$ we set
$T_{0}:=t_{0}-\frac{\beta_{1}}{K}$. For $\overline{t}:=t-T_{0}$ we
let $\overline{g}(\overline{t})$ and $\overline{H}(\overline{t})$ be the solution of the system
\begin{eqnarray*}
\frac{\partial}{\partial\overline{t}}\bar{g}&=&-2\overline{{\rm
Rc}}+\frac{1}{2}\overline{h}, \\
\frac{\partial}{\partial\overline{t}}\overline{H}&=&\Delta_{{\rm HL},\overline{g}}\overline{H},
\ \ \ \overline{g}(0)=g(T_{0}), \ \ \ \overline{H}(0)=H(T_{0}).
\end{eqnarray*}
The uniqueness of solution implies that
$\overline{g}(\overline{t})=g(\overline{t}+T_{0})=g(t)$ for
$\overline{t}\in[0,\beta_{1}/K]$. By the assumption we have
\begin{equation*}
|\overline{\rm
Rm}(x,\overline{t})|_{\overline{g}(x,\overline{t})}\leq K, \ \ \
|\overline{H}(x)|_{\overline{g}(x)}\leq K
\end{equation*}
for all $x\in M$ and $\overline{t}\in[0,\beta_{1}/K]$.
Applying Theorem 4.5 with $\alpha=\beta_{1}$, we have
\begin{equation*}
|\overline{\nabla}^{m-1}\overline{\rm
Rm}(x,\overline{t})|_{\overline{g}(x,\overline{t})}+
|\overline{\nabla}^{m}H(x,\overline{t})|_{\overline{g}(x,\overline{t})}\leq\frac{\overline{C}_{m}}{\overline{t}^{m/2}}
\end{equation*}
for all $x\in M$ and $\overline{t}\in(0,\beta_{1}/K]$. We
have $\overline{t}^{m/2}\geq\beta^{m/2}_{1}2^{-m/2}K^{-m/2}$ if
$\overline{t}\in[\beta_{1}/2K,\beta_{1}/K]$. Taking
$\overline{t}=\beta_{1}/K$, we obtain
\begin{equation*}
|\nabla^{m-1}{\rm
Rm}(x,t_{0})|_{g(x,t_{0})}+|\nabla^{m}H(x,t_{0})|_{g(x,t_{0})}
\leq\frac{2^{m/2}\overline{C}_{m}K^{m/2}}{\beta^{m/2}_{1}}
\end{equation*}
for all $x\in M$. Since $t_{0}\in[\beta/K,T]$ was arbitrary,
the result follows.
\end{proof}

\section{Compactness theorem}\label{5}
In this section we prove the compactness theorem for our generalized
Ricci flow. We follow Hamilton's method \cite{H2} on the compactness
theorem for the usual Ricci flow.

We review several definitions from \cite{CCGGIIKLLN}. Throughout this
section, all Riemannian manifolds are smooth manifolds of dimensions $n$. The covariant derivative with respect to a metric
$g$ will be denoted by ${}^{g}\nabla$.

\begin{definition} \label{d5.1}Let $K\subset M$ be a compact set and let
$\{g_{k}\}_{k\in\mathbb{N}}, g_{\infty}$, and $g$ be Riemannian metrics on
$M$. For $p\in\{0\}\cup\mathbb{N}$ we say that $g_{k}$ {\it converges in
$C^{p}$ to $g_{\infty}$ uniformly on $K$ with respect to $g$} if for
every $\epsilon>0$ there exists $k_{0}=k_{0}(\epsilon)>0$ such that
for $k\geq k_{0}$,
\begin{equation}
\|g_{k}-g_{\infty}\|_{C^{p};K,g}:=\sup_{0\leq\alpha\leq p}\sup_{x\in
K}|{}^{g}\nabla^{\alpha}(g_{k}-g_{\infty})(x)|_{g}<\epsilon.\label{5.1}
\end{equation}
\end{definition}

Since we consider a compact set, the choice of background metric $g$
does not change the convergence. Hence we may choose $g=g_{\infty}$.

\begin{definition} \label{d5.2}Suppose $\{U_{k}\}_{k\in\mathbb{N}}$ is an exhaustion\footnote{If for any compact set $K\subset M$ there exists
$k_{0}\in\mathbb{N}$ such that $U_{k}\supset K$ for all $k\geq k_{0}$} of a
smooth manifold $M$ by open sets and $g_{k}$ are Riemannian metrics
on $U_{k}$. We say that $(U_{k},g_{k})$ {\it converges in
$C^{\infty}$ to $(M,g_{\infty})$ uniformly on compact sets in $M$}
if for any compact set $K\subset M$ and any $p>0$ there exists
$k_{0}=k_{0}(K,p)$ such that $\{g_{k}\}_{k\geq k_{0}}$ converges in
$C^{p}$ to $g_{\infty}$ uniformly on $K$.
\end{definition}

A {\it pointed Riemannian manifold} is a $3$-tuple $(M,g,O)$, where
$(M,g)$ is a Riemannian manifold and $O\in M$ is a basepoint. If the
metric $g$ is complete, the $3$-tuple is called a {\it complete
pointed Riemannian manifold}. We say $(M,g(t),H(t),O),
t\in(\alpha,\omega)$, is a {\it pointed solution to the generalized
Ricci flow} if $(M,g(t),H(t))$ is a solution to the generalized
Ricci flow.
\\

The so-called {\it Cheeger-Gromov convergence} in $C^{\infty}$ is
defined by

\begin{definition} \label{d5.3} A given
sequence $\{(M_{k},g_{k},O_{k})\}_{k\in\mathbb{N}}$ of complete pointed
Riemannian manifolds {\it converges} to a complete pointed
Riemannian manifold $(M_{\infty},g_{\infty},O_{\infty})$ if there
exist
\begin{itemize}

\item[(i)] an exhaustion $\{U_{k}\}_{k\in\mathbb{N}}$ of $M_{\infty}$ by
open sets with $O_{\infty}\in U_{k}$,

\item[(ii)] a sequence of diffeomorphisms $\Phi_{k}: M_{\infty}\ni U_{k}\to
V_{k}:=\Phi_{k}(U_{k})\subset M_{k}$ with
$\Phi_{k}(O_{\infty})=O_{k}$

\end{itemize}
such that $(U_{k},\Phi^{\ast}_{k}(g_{k}|_{V_{k}}))$ converges in
$C^{\infty}$ to $(M_{\infty},g_{\infty})$ uniformly on compact sets
in $M_{\infty}$.
\end{definition}

The corresponding convergence for the generalized Ricci flow is
similar to the convergence for the usual Ricci flow introduced by
Hamilton \cite{H2}.

\begin{definition} \label{d5.4}A given sequence $\{(M_{k},g_{k}(t),H_{k}(t),O_{k})\}_{k
\in\mathbb{N}}$ of complete pointed
solutions to the GRF {\it converges} to a complete pointed solution
to the GRF
\begin{equation*}
(M_{\infty},g_{\infty}(t),H_{\infty}(t),O_{\infty}),
t\in(\alpha,\omega),
\end{equation*}
if there exist
\begin{itemize}

\item[(i)] an exhaustion $\{U_{k}\}_{k\in\mathbb{N}}$ of $M_{\infty}$ by
open sets with $O_{\infty}\in U_{k}$,

\item[(ii)] a sequence of diffeomorphisms $\Phi_{k}: M_{\infty}\ni U_{k}\to
V_{k}:=\Phi_{k}(U_{k})\subset M_{k}$ with
$\Phi_{k}(O_{\infty})=O_{k}$

\end{itemize}
such that
$(U_{k}\times(\alpha,\omega),\Phi^{\ast}_{k}(g_{k}(t)|_{V_{k}})+dt^{2},\Phi^{\ast}_{k}(H_{k}(t)|_{V_{k}}))$
converges in $C^{\infty}$ to
$(M_{\infty}\times(\alpha,\omega),g_{\infty}(t)+dt^{2},H_{\infty}(t))$
uniformly on compact sets in $M_{\infty}\times(\alpha,\omega)$. Here
we denote by $dt^{2}$ the standard metric on $(\alpha,\omega)$.
\end{definition}

Let ${\rm inj}_{g}(O)$ be the injectivity radius of the metric $g$
at the point $O$. The following compactness theorem is due to Cheeger and Gromov.

\begin{theorem} \label{t5.5} {\bf (Compactness for metrics)} Let $\{(M_{k},g_{k},O_{k})\}_{k\in\mathbb{N}}$ be a sequence of
complete pointed Riemannian manifolds satisfying

\begin{itemize}

\item[(i)] for all $p\geq0$ and $k\in\mathbb{N}$, there is a sequence of
constants $C_{p}<\infty$ independent of $k$ such that
\begin{equation*}
|{}^{g_{k}}\nabla^{p}{\rm Rm}(g_{k})|_{g_{k}}\leq C_{p}
\end{equation*}
on $M_{k}$,

\item[(ii)] there exists some constant $\iota_{0}>0$ such that
\begin{equation*}
{\rm inj}_{g_{k}}(O_{k})\geq\iota_{0}
\end{equation*}
for all $k\in\mathbb{N}$.

\end{itemize}
Then there exists a subsequence $\{j_{k}\}_{k\in\mathbb{N}}$ such that
$\{(M_{j_{k}},g_{j_{k}},O_{j_{k}})\}_{k\in\mathbb{N}}$ converges to a
complete pointed Riemannaian manifold
$(M^{n}_{\infty},g_{\infty},O_{\infty})$ as $k\to\infty$.
\end{theorem}

As a consequence of Theorem \ref{t5.5}, we state our
compactness theorem for GRF.

\begin{theorem} \label{t5.6}{\bf (Compactness for GRF)} Let
$\{(M_{k},g_{k}(t),H_{k}(t),O_{k})\}_{k\in\mathbb{N}}$ be a sequence of
complete pointed solutions to {\rm GRF} for
$t\in[\alpha,\omega)\ni0$ such that

\begin{itemize}

\item[(i)] there is a constant $C_{0}<\infty$ independent of
$k$ such that
\begin{equation*}
\sup_{(x,t)\in M_{k}\times(\alpha,\omega)}|{\rm
Rm}(g_{k}(x,t))|_{g_{k}(x,t)}\leq C_{0}, \ \ \ \sup_{x\in
M_{k}}|H_{k}(x,\alpha)|_{g_{k}(x,\alpha)}\leq C_{0},
\end{equation*}

\item[(ii)] there exists a constant $\iota_{0}>0$ satisfying
\begin{equation*}
{\rm inj}_{g_{k}(0)}(O_{k})\geq\iota_{0}.
\end{equation*}

\end{itemize}
Then there exists a subsequence $\{j_{k}\}_{k\in\mathbb{N}}$ such that
\begin{equation*}
(M_{j_{k}},g_{j_{k}}(t),H_{j_{k}}(t),O_{j_{k}})\longrightarrow
(M_{\infty},g_{\infty}(t),H_{\infty}(t),O_{\infty}),
\end{equation*}
converges to a complete pointed solution
$(M_{\infty},g_{\infty}(t),H_{\infty}(t),O_{\infty}),
t\in[\alpha,\omega)$ to {\rm GRF}
 as $k\to\infty$.
\end{theorem}

To prove Theorem \ref{t5.6} we extend a lemma
for Ricci flow to {\rm GRF}. After establishing this lemma, the
proof of Theorem \ref{t5.6} is similar to Theorem 3.10 in
\cite{CCGGIIKLLN}.

\begin{lemma} \label{l5.7}Let $(M,g)$ be a Riemannian manifold with a background
metric $g$, let $K$ be a compact subset of $M$, and let
$(g_{k}(x,t),H_{k}(x,t))$ be a collection of solutions to the
generalized Ricci flow defined on neighborhoods of
$K\times[\beta,\psi]$, where $t_{0}\in[\beta,\psi]$ is a fixed time.
Suppose that
\begin{itemize}

\item[(i)] the metrics $g_{k}(x,t_{0})$ are all uniformly equivalent
to $g(x)$ on $K$, i.e., for all $V\in T_{x}M, k$, and $x\in K$,
\begin{equation*}
C^{-1}g(x)(V,V)\leq g_{k}(x,t_{0})(V,V)\leq Cg(x)(V,V),
\end{equation*}
where $C<\infty$ is a constant independent of $V,k$, and $x$,

\item[(ii)] the covariant derivatives of the metrics $g_{k}(x,t_{0})$ with respect to the metric $g(x)$
are all uniformly bounded on $K$, i.e., for all $k$ and $p\geq1$,
\begin{equation*}
|{}^{g}\nabla^{p}g_{k}(x,t_{0})|_{g(x)}+|{}^{g}\nabla^{p-1}H_{k}(x,t_{0})|_{g(x)}\leq
C_{p}
\end{equation*}
where $C_{p}<\infty$ is a sequence of constants independent of $k$,

\item[(iii)] the covariant derivatives of the curvature tensors
${\rm Rm}(g_{k}(x,t))$ and of the forms $H_{k}(x,t)$ are uniformly
bounded with respect to the metric $g_{k}(x,t)$ on
$K\times[\beta,\psi]$, i.e., for all $k$ and $p\geq0$,
\begin{equation*}
|{}^{g_{k}}\nabla^{p}{\rm
Rm}(g_{k}(x,t))|_{g_{k}(x,t)}+|{}^{g_{k}}\nabla^{p}H_{k}(x,t)|_{g_{k}(x,t)}\leq
C'_{p}
\end{equation*}
where $C'_{p}$ is a sequence of constants independent of $k$.
\end{itemize}
Then the metrics $g_{k}(x,t)$ are uniformly equivalent to $g(x)$ on
$K\times[\beta,\psi]$, i.e.,
\begin{equation*}
B(t,t_{0})^{-1}g(x)(V,V)\leq g_{k}(x,t)(V,V)\leq
B(t,t_{0})g(x)(V,V),
\end{equation*}
where $B(t,t_{0})=Ce^{C'_{0}|t-t_{0}|}$(Here the constant $C'_{0}$
may not be equal to the previous one), and the time-derivatives and
covariant derivatives of the metrics $g_{k}(x,t)$ with respect to
the metric $g(x)$ are uniformly bounded on $K\times[\beta,\psi]$,
i.e., for each $(p,q)$ there is a constant $\widetilde{C}_{p,q}$
independent of $k$ such that
\begin{equation*}
\left|\frac{\partial^{q}}{\partial
t^{q}}{}^{g}\nabla^{p}g_{k}(x,t)\right|_{g(x)}+\left|\frac{\partial^{q}}{\partial
t^{q}}{}^{g}\nabla^{p-1}H_{k}(x,t)\right|_{g(x)}\leq\widetilde{C}_{p,q}
\end{equation*}
for all $k$.
\end{lemma}

\begin{proof} Before proving the lemma, we quote a fact
from \cite{CCGGIIKLLN}, i.e., Lemma 3.13:
Suppose that the metrics $g_{1}$ and $g_{2}$ are equivalent, i.e.,
$C^{-1}g_{1}\leq g_{2}\leq Cg_{1}$. Then for any $(p,q)$-tensor $T$
we have $|T|_{g_{2}}\leq C^{(p+q)/2}|T|_{g_{1}}$. We denote by $h$
the tensor $h_{ij}:=g^{kp}g^{lq}H_{ikl}H_{jpq}$. In the following we
denote by $C$ a constant depending only on $n, \beta$, and $\psi$
which may take different values at different places. For any tangent
vector $V\in T_{x}M$ we have
\begin{equation*}
\frac{\partial}{\partial t}g_{k}(x,t)(V,V)=-2{\rm
Ric}(g_{k}(x,t))(V,V)+\frac{1}{2}h_{k}(x,t)(V,V),
\end{equation*}
and therefore
\begin{eqnarray*}
\left|\frac{\partial}{\partial t}{\rm
log}g_{k}(x,t)(V,V)\right|&=&\left|\frac{-2{\rm Ric}(g_{k}(x,t))(V,V)+\frac{1}{2}h_{k}(x,t)(V,V)}{g_{k}(x,t)(V,V)}\right|
\\
&\leq&C'_{0}+C|H_{k}(x,t)|^{2}_{g_{k}(x,t)} \\
&\leq&C'_{0}+CC'^{2}_{0}:=\overline{C},
\end{eqnarray*}
since
\begin{eqnarray*}
|{\rm Ric}(g_{k}(x,t))(V,V)|&\leq&C'_{0}g_{k}(x,t)(V,V), \\
|h_{k}(x,t)(V,V)|&\leq&C|H_{k}(x,t)|^{2}_{g_{k}(x,t)}g_{k}(x,t)(V,V).
\end{eqnarray*}
Integrating on both sides, we have
\begin{eqnarray*}
\overline{C}|t_{1}-t_{0}|&\geq&\int^{t_{1}}_{t_{0}}\left|\frac{\partial}{\partial
t}{\rm log} g_{k}(x,t)(V,V)\right|dt \\
&\geq&\left|\int^{t_{1}}_{t_{0}}\frac{\partial}{\partial t}{\rm
log}g_{k}(t)(V,V)dt\right| \\
&=&\left|{\rm
log}\frac{g_{k}(x,t_{1})(V,V)}{g_{k}(x,t_{0})(V,V)}\right|
\end{eqnarray*}
and hence we conclude that
\begin{equation*}
e^{-\overline{C}|t_{1}-t_{0}|}g_{k}(x,t_{0})(V,V)\leq
g_{k}(x,t_{1})(V,V)\leq
e^{\overline{C}|t_{1}-t_{0}|}g_{k}(x,t_{0})(V,V).
\end{equation*}
From the assumption (i), it immediately deduces from above that
\begin{equation*}
C^{-1}e^{-\overline{C}|t_{1}-t_{0}|}g(x)(V,V)\leq
g_{k}(x,t_{1})(V,V)\leq Ce^{\overline{C}|t_{1}-t_{0}|}g(x)(V,V).
\end{equation*}
Since $t_{1}$ was arbitrary, the first part is proved. From the
definition(or see \cite{CCGGIIKLLN}, eq. (37), page 134), we have
\begin{equation*}
(g_{k})^{ec}({}^{g}\nabla_{a}(g_{k})_{bc}+{}^{g}\nabla_{b}(g_{k})_{ac}
-{}^{g}\nabla_{c}(g_{k})_{ab})=2({}^{g_{k}}\Gamma)^{e}_{ab}-2({}^{g}\Gamma)^{e}_{ab}.
\end{equation*}
Thus $|{}^{g_{k}}\Gamma(x,t)-{}^{g}\Gamma(x)|_{g(x)}\leq
C|{}^{g}\nabla g_{k}(x,t)|_{g_{k}(x)}$. On the other hand,
\begin{equation*}
{}^{g}\nabla_{a}(g_{k})_{bc}=(g_{k})_{eb}[({}^{g_{k}}\Gamma)^{e}_{ac}-({}^{g}\Gamma)^{e}_{ac}]+(g_{k})_{ec}[({}^{g_{k}}\Gamma)^{e}_{ab}-({}^{g}\Gamma)^{e}_{ab}],
\end{equation*}
it follows that $|{}^{g}\nabla g_{k}(x,t)|_{g_{k}(x,t)}\leq
C|{}^{g_{k}}\Gamma(x,t)-{}^{g}\Gamma(x)|_{g_{k}(x,t)}$ and therefore
\begin{equation}
{}^{g}\nabla g_{k} \ \text{is equivalent to} \ {}^{g_{k}}\Gamma-{}^{g}\Gamma={}^{g_{k}}\nabla-{}^{g}\nabla. \label{5.2}
\end{equation}
The
evolution equation for ${}^{g}\Gamma$ is
\begin{eqnarray*}
\frac{\partial}{\partial
t}({}^{g_{k}}\Gamma)^{c}_{ab}&=&-(g_{k})^{cd}[({}^{g_{k}}\nabla)_{a}({\rm Ric}(g_{k}))_{bd}+({}^{g_{k}}\nabla)_{b}({\rm Ric}(g_{k}))_{ad} \\
&&- \ ({}^{g_{k}}\nabla)_{d}({\rm Ric}(g_{k}))_{ab}] \\
&&+ \ \frac{1}{4}(g_{k})^{cd}[({}^{g_{k}}\nabla)_{a}(h_{k})_{bd}+({}^{g_{k}}\nabla)_{b}(h_{k})_{ad}-({}^{g_{k}}\nabla)_{d}(h_{k})_{ab}].
\end{eqnarray*}
Since ${}^{g}\Gamma$ does not depend on $t$, it follows from the
assumptions that
\begin{eqnarray*}
\left|\frac{\partial}{\partial
t}({}^{g_{k}}\Gamma-{}^{g}\Gamma)\right|_{g_{k}}&\leq&C|{}^{g_{k}}\nabla({\rm Ric}(g_{k}))|_{g_{k}}+C|{}^{g_{k}}\nabla(h_{k})|_{g_{k}} \\
&\leq&CC'_{1}+C|{}^{g_{k}}\nabla H_{k}|_{g_{k}}\cdot|H_{k}|_{g_{k}}
\ \leq \ C'_{1}.
\end{eqnarray*}
Integrating on both sides,
\begin{eqnarray*}
C'_{1}|t_{1}-t_{0}|&\geq&\left|\int^{t_{1}}_{t_{0}}\frac{\partial}{\partial
t}({}g^{k}\Gamma(t)-{}^{g}\Gamma)dt\right|_{g_{k}}
\\
&\geq&|{}^{g_{k}}\Gamma(t_{1})-{}^{g}\Gamma|_{g_{k}}-|{}^{g_{k}}\Gamma(t_{0})-{}^{g}\Gamma|_{g_{k}}.
\end{eqnarray*}
Hence we obtain
\begin{eqnarray*}
|{}^{g_{k}}\Gamma(t)-{}^{g}\Gamma|_{g_{k}}&\leq&C'_{1}|t_{1}-t_{0}|+|{}^{g_{k}}\Gamma(t_{0})-{}^{g}\Gamma|_{g_{k}}
\\
&\leq&C'_{1}|t_{1}-t_{0}|+C|{}^{g}\nabla g_{k}(t_{0})|_{g_{k}} \\
&\leq&C'_{1}|t-t_{0}|+C|{}^{g}\nabla g_{k}(t_{0})|_{g} \\
&\leq&C'_{1}|t-t_{0}|+C_{1}.
\end{eqnarray*}
The equivalency of metrics tells us that
\begin{eqnarray*}
|{}^{g}\nabla g_{k}(t)|_{g}&\leq& B(t,t_{0})^{3/2}|{}^{g}\nabla
g_{k}(t)|_{g_{k}}  \ \leq \ B(t,t_{0})^{3/2}\cdot C|{}^{g_{k}}\Gamma(t)-{}^{g}\Gamma|_{g_{k}}\\
&\leq& B(t,t_{0})^{3/2}(C'_{1}|t-t_{0}|+C').
\end{eqnarray*}
Since $|t-t_{0}|\leq\psi-\beta$, it follows that $|{}^{g}\nabla
g_{k}(t)|_{g}\leq\widetilde{C}_{1,0}$ for some constant
$\widetilde{C}_{1,0}$. But $g$ and $g_{k}$ are equivalent, we have
\begin{equation*}
|H_{k}(t)|_{g}\leq C|H_{k}(t)|_{g_{k}}\leq
CC'_{1}=\widetilde{C}_{1,0}.
\end{equation*}
From the assumptions, we also have
\begin{eqnarray*}
|{}^{g}\nabla H_{k}|_{g}&\leq&|({}^{g}\nabla-{}^{g_{k}}\nabla)H_{k}+{}^{g_{k}}\nabla H_{k}|_{g} \\
&\leq&C|{}^{g}\nabla g_{k}|_{g}\cdot|H_{k}|_{g}+C|{}^{g_{k}}\nabla
H_{k}|_{g_{k}} \\
&\leq&CC'_{1}+C\widetilde{C}_{1,0}\widetilde{C}_{1,0}=:\widetilde{C}_{2,0}.
\end{eqnarray*}
Moreover,
\begin{eqnarray*}
\frac{\partial}{\partial t}{}^{g}\nabla
H_{k}&=&{}^{g}\nabla(\Delta_{g_{k}}H_{k}+{\rm Rm}(g_{k})\ast H_{k})
\\
&=&({}^{g}\nabla-{}^{g_{k}}\nabla)\Delta_{g_{k}}H_{k}+{}^{g_{k}}\nabla\Delta_{g_{k}}H_{k}
\\
&&+ \ {}^{g}\nabla{\rm Rm}(g_{k})\ast H_{k}+{\rm
Rm}(g_{k})\ast{}^{g}\nabla H_{k}
\end{eqnarray*}
where $\Delta_{g_{k}}$ is the Laplace operator associated to
$g_{k}$. Hence
\begin{eqnarray*}
\left|\frac{\partial}{\partial t}{}^{g}\nabla
H_{k}\right|_{g}&\leq&C|{}^{g}\nabla g_{k}|_{g}\cdot|\Delta_{g_{k}}
H_{k}|_{g_{k}}+C|{}^{g_{k}}\nabla\Delta_{g_{k}}H_{k}|_{g} \\
&&+ \ C|{}^{g}\nabla{\rm Rm}(g_{k})|_{g}\cdot|H_{k}|_{g}+C|{\rm
Rm}(g_{k})|_{g}\cdot|{}^{g}\nabla H_{k}|_{g}\\
&\leq&\widetilde{C}_{2,1}.
\end{eqnarray*}
For higher derivatives we claim that
\begin{equation}
|{}^{g}\nabla^{p}{\rm Ric}(g_{k})|_{g}\leq
C''_{p}|{}^{g}\nabla^{p}g_{k}|_{g}+C'''_{p}, \ \ \
|{}^{g}\nabla^{p}g_{k}|_{g}+|{}^{g}\nabla^{p-1}H_{k}|_{g}
\leq\widetilde{C}_{p,0},\label{5.3}
\end{equation}
for all $p\geq1$, where $C''_{p}, C'''_{p}$, and
$\widetilde{C}_{p,0}$ are constants independent of $k$. For $p=1$,
we have proved the second inequality, so we suffice to prove the
first one with $p=1$. Indeed,
\begin{eqnarray*}
|{}^{g}\nabla{\rm Ric}(g_{k})|_{g}&\leq&
C|({}^{g}\nabla-{}^{g_{k}}\nabla){\rm Ric}(g_{k})+{}^{g_{k}}\nabla{\rm Ric}(g_{k})|_{g_{k}} \\
&\leq& C|{}^{g}\Gamma-{}^{g_{k}}\Gamma|_{g}\cdot|{\rm Ric}(g_{k})|_{g_{k}}+C|{}^{g_{k}}\nabla{\rm Ric}(g_{k})|_{g_{k}} \\
&\leq&C''_{1}|{}^{g}\nabla g_{k}|_{g}+C'''_{1}.
\end{eqnarray*}
Suppose the claim holds for all $p<N$ ($N\geq2$), we shall show that
it also holds for $p=N$. From
\begin{eqnarray*}
|{}^{g}\nabla^{N}{\rm Ric}(g_{k})|_{g}&=&\left|\sum^{N}_{i=1}{}^{g}\nabla^{N-i}({}^{g}\nabla-{}^{g_{k}}\nabla){}^{g_{k}}\nabla^{i-1}{\rm Ric}(g_{k})+{}^{g_{k}}\nabla^{N}{\rm Ric}(g_{k})\right|_{g} \\
&\leq&\sum^{N}_{i=1}\left|{}^{g}\nabla^{N-i}({}^{g}\nabla-{}^{g_{k}}\nabla){}^{g_{k}}\nabla^{i-1}{\rm Ric}(g_{k})\right|_{g}+|{}^{g_{k}}\nabla^{N}{\rm Ric}(g_{k})|_{g}
\end{eqnarray*}
we estimate each term. For $i=1$, by induction and assumptions we
have
\begin{eqnarray*}
& & |{}^{g}\nabla^{N-1}({}^{g}\nabla-{}^{g_{k}}\nabla){\rm Ric}(g_{k})|_{g} \\
&\leq&C|{}^{g}\nabla^{N-1}({}^{g}\nabla
g_{k}\cdot{\rm Ric}(g_{k}))|_{g} \\
&\leq&
C\left|\sum^{N-1}_{j=0}\binom{N-1}{j}{}^{g}\nabla^{N-1-j}({}^{g}\nabla
g_{k})\cdot{}^{g}\nabla^{j}({\rm Ric}(g_{k}))\right|_{g} \\
&\leq&C\sum^{N-1}_{j=0}\binom{N-1}{j}|{}^{g}\nabla^{N-j}g_{k}|_{g}\cdot|{}^{g}\nabla^{j}{\rm Ric}(g_{k})|_{g} \\
&\leq&C\sum^{N-1}_{i=0}\binom{N-1}{j}(C''_{j}|{}^{g}\nabla^{j}g_{k}|_{g}+C'''_{j})|{}^{g}\nabla^{N-j}g_{k}|_{g}
\\
&\leq&C\sum^{N-1}_{j=0}\binom{N-1}{j}(C''_{j}\widetilde{C}_{j,0}+C'''_{j})|{}^{g}\nabla^{N-j}g_{k}|_{g}
\\
&=&C(N-1)(C''_{0}\widetilde{C}_{j,0}+C'''_{0})|{}^{g}\nabla^{N}g_{k}|_{g}
\\
&&+ \ C\sum^{N-1}_{j=1}\binom{N-1}{j}(C''_{j}\widetilde{C}_{j,0}+C'''_{j})\widetilde{C}_{N-j,0}
\\
&\leq&C''_{N}|{}^{g}\nabla^{N}g_{k}|_{g}+C'''_{N}.
\end{eqnarray*}
For $i\geq2$, we have
\begin{eqnarray*}
& &
|{}^{g}\nabla^{N-i}({}^{g}\nabla-{}^{g_{k}}\nabla){}^{g_{k}}\nabla^{i-1}{\rm Ric}(g_{k})|_{g} \\
&\leq&C|{}^{g}\nabla^{N-i}({}^{g}\nabla
g_{k}\cdot{}^{g_{k}}\nabla^{i-1}{\rm Ric}(g_{k}))|_{g} \\
&\leq&
C\sum^{N-i}_{j=0}\binom{N-i}{j}|{}^{g}\nabla^{N-i-j+1}g_{k}|_{g}\cdot|{}^{g}\nabla^{j}\cdot{}^{g_{k}}\nabla^{i-1}{\rm Ric}(g_{k})|_{g}.
\end{eqnarray*}
If $j=0$, then
\begin{equation*}
|{}^{g_{k}}\nabla^{i-1}{\rm Ric}(g_{k})|_{g}\leq
C''_{i-1}|{}^{g}\nabla^{i-1}g_{k}|_{g}+C'''_{i-1}\leq
C''_{i-1}\widetilde{C}_{i-1,0}+C'''_{i-1}.
\end{equation*}
Suppose in the following that $j\geq1$. Hence
\begin{eqnarray*}
|{}^{g}\nabla^{j}\cdot{}^{g_{k}}\nabla^{i-1}{\rm Ric}(g_{k})|_{g}&=&\left|(({}^{g}\nabla-{}^{g_{k}}\nabla)+{}^{g_{k}}\nabla)^{j}\cdot
{}^{g_{k}}\nabla^{i-1}{\rm Ric}(g_{k})\right|_{g} \\
&\leq&C\sum^{j}_{l=0}\binom{j}{l}|{}^{g}\nabla^{l}g_{k}|_{g}\cdot|{}^{g_{k}}\nabla^{j-l+i-1}{\rm Ric}(g_{k})|_{g} \\
&\leq&C\sum^{j}_{l=0}\binom{j}{l}\widetilde{C}_{l,0}
(C''_{j-l+i-1}\widetilde{C}_{j-l+i-1,0}+C'''_{j-l+i-1}),
\end{eqnarray*}
where we make use of (\ref{5.2}) from first line to second line. Combining these inequalities, we get
\begin{equation*}
|{}^{g}\nabla^{N}{\rm Ric}(g_{k})|_{g}\leq
C''_{N}|{}^{g}\nabla^{N}g_{k}|_{g}+C'''_{N}.
\end{equation*}
Similarly, we have
\begin{equation*}
|{}^{g}\nabla^{N}h_{k}|_{g}\leq
C''_{N}|{}^{g}\nabla^{N}g_{k}|_{g}+C'''_{N}.
\end{equation*}
Since $\frac{\partial}{\partial t}g_{k}=-2{\rm Ric}(g_{k})+\frac{1}{2}h_{k}$, it follows that
\begin{eqnarray*}
\frac{\partial}{\partial
t}{}^{g}\nabla^{N}g_{k}&=&{}^{g}\nabla^{N}(-2{\rm Ric}(g_{k})+\frac{1}{2}h_{k}) \\
\frac{\partial}{\partial
t}|{}^{g}\nabla^{N}g_{k}|^{2}_{g}&\leq&\left|\frac{\partial}{\partial
t}{}^{g}\nabla^{N}g_{k}\right|^{2}_{g}+|{}^{g}\nabla^{N}g_{k}|^{2}_{g}
\\
&\leq&8|{}^{g}\nabla^{N}{\rm Ric}(g_{k})|^{2}_{g}+\frac{1}{2}|{}^{g}\nabla^{N}h_{k}|^{2}_{g}+|{}^{g}\nabla^{N}g_{k}|^{2}_{g}
\\
&\leq&(1+18(C''_{N})^{2})|{}^{g}\nabla^{N}g_{k}|^{2}_{g}+18(C''_{N})^{2}.
\end{eqnarray*}
Integrating the above inequality, we get $|{}^{g}\nabla
g_{k}|_{g}\leq\widetilde{C}_{N,0}$ and therefore
$|{}^{g}\nabla^{N}h_{k}|_{g}\leq\widetilde{C}_{N+1,0}$. We have
proved lemma for $q=0$. When $g\geq1$, then
\begin{equation*}
\frac{\partial^{q}}{\partial
t^{q}}{}^{g}\nabla^{p}g_{k}(t)={g}\nabla^{p}\frac{\partial^{q-1}}{\partial
t^{q-1}}\left(-2{\rm Ric}(g_{k}(t))+\frac{1}{2}h_{k}(t)\right).
\end{equation*}
Using the evolution equations for ${\rm Rm}(g_{k}(t))$ and
$h_{k}(t)$, combining the induction to $q$ and using the above
method, we have $|\frac{\partial^{q}}{\partial
t^{q}}{}^{g}\nabla^{p}g_{k}(t)|_{g}+|\frac{\partial^{q}}{\partial
t^{q}}{}^{g}\nabla^{p-1}h_{k}(t)|_{g}\leq\widetilde{C}_{p,q}$.
\end{proof}

\section{Generalization}\label{6}

In this section, we generalize the main results in Sec. 4 and Sec. 5
to a kind of generalized Ricci flow in which the local existence has
been established \cite{HHKL}.

Let $(M,g_{ij}(x))$ be an $n$-dimensional closed Riemannian
manifold and let $A=\{A_{i}\}$ and $B=\{B_{ij}\}$ denote a one form
and two form respectively. Set $F=dA$ and $H=dB$. The authors in \cite{HHKL}
proved that there exists a constant $T>0$ such that the evolution
equations
\begin{eqnarray*}
\frac{\partial}{\partial
t}g_{ij}(x,t)&=&-2R_{ij}(x,t)+\frac{1}{2}h_{ij}(x,t)+2f_{jk}(x,t),
\ \ \ g_{ij}(x,0)=g_{ij}(x),\\
\frac{\partial}{\partial t}A_{i}(x,t)&=&-2\nabla_{k}F_{i}{}^{k}(x,t), \ \ \ A_{i}(x,0)=A_{i}(x),\\
\frac{\partial}{\partial
t}B_{ij}(x,t)&=&3\nabla_{k}H^{k}{}_{ij}(x,t), \ \ \
B_{ij}(x,0)=B_{ij}(x)
\end{eqnarray*}
has a unique smooth solution on $m\times[0,T)$, where
$h_{ij}=H_{ikl}H_{j}{}^{kl}$ and $f_{ij}=F_{i}{}^{k}F_{jk}$.
We call it as ${\rm RF}(A,B)$. According to the definition of the adjoint
operator $d^{\ast}$, we have
\begin{equation}
(d^{\ast}F)_{i}=2\nabla_{k}F_{i}{}^{k}, \ \ \
(d^{\ast}H)_{ij}=-3\nabla_{k}H^{k}{}_{ij},\label{6.1}
\end{equation}
and hence
\begin{eqnarray}
\frac{\partial}{\partial t}F(x,t)&=&-dd^{\ast}_{g(x,t)}F \ = \
\Delta_{{\rm HL},g(x,t)}F \ = \ \Delta F+{\rm Rm}\ast F, \label{6.2}\\
\frac{\partial}{\partial t}H(x,t)&=&-dd^{\ast}_{g(x,t)}H \ = \
\Delta_{{\rm HL},g(x,t)}H \ = \ \Delta H+{\rm Rm}\ast H.\label{6.3}
\end{eqnarray}

They also derived the evolution equations of curvatures:
\begin{eqnarray*}
\frac{\partial}{\partial t}R_{ijk\ell}&=&\Delta
R_{ijk\ell}+2(B_{ijk\ell}-B_{ij\ell k}-B_{i\ell jk}+B_{ikj\ell}) \\
&&- \ g^{pq}(R_{pjk\ell}R_{qi}+R_{ipk\ell}R_{qj}+R_{ijp\ell}
R_{qk}+R_{ijkp}R_{q\ell})
\\
&&+ \ \frac{1}{4}[\nabla_{i}\nabla_{\ell}(H_{kpq}H_{j}{}^{pq})
-\nabla_{i}\nabla_{k}(H_{jpq}H_{\ell}{}^{pq})
\\
&&- \ \nabla_{j}\nabla_{\ell}(H_{kpq}H_{i}{}^{pq})
+\nabla_{j}\nabla_{k}(H_{ipq}H_{\ell}{}^{pq})]
\\
&&+ \ \frac{1}{4}g^{rs}(H_{kpq}H_{r}^{pq}R_{ijs\ell}
+H_{rpq}H_{\ell}{}^{pq}R_{ijks})
\\
&&+ \ \nabla_{i}\nabla_{\ell}(F_{k}{}^{p}F_{jp})-\nabla_{i}
\nabla_{k}(F_{j}{}^{p}F_{\ell p})
-\nabla_{j}\nabla_{\ell}(F_{k}{}^{p}F_{ip})
+\nabla_{j}\nabla_{k}(F_{i}{}^{p}F_{\ell p})
\\
&&+ \ g^{rs}(F_{k}{}^{p}F_{rp}R_{ijs\ell}+F_{r}{}^{p}F_{\ell p}R_{ijks}).
\end{eqnarray*}
Under our notation, it can be rewritten as
\begin{eqnarray}
\frac{\partial}{\partial t}{\rm Rm}&=&\Delta{\rm
Rm}+\sum_{i+j=0}\nabla^{i}{\rm Rm}\ast\nabla^{j}{\rm
Rm}+\sum_{i+j=0+2}\nabla^{i}H\ast\nabla^{j}H \nonumber\\
&+&\sum_{i+j=0+2}\nabla^{i}F\ast\nabla^{j}F+\sum_{i+j+k=0}\nabla^{i}H\ast\nabla^{j}H\ast\nabla^{k}{\rm
Rm} \label{6.4}\\
&+&\sum_{i+j+k=0}\nabla^{i}F\ast\nabla^{j}F\ast\nabla^{k}{\rm Rm}.\nonumber
\end{eqnarray}

As before, we have

\begin{proposition} \label{p6.1}For ${\rm RF}(A,B)$ and any
nonnegative integer $\ell$ we
have
\begin{eqnarray}
\frac{\partial}{\partial t}\nabla^{\ell}{\rm
Rm}&=&\Delta(\nabla^{l}{\rm Rm})+\sum_{i+j=\ell}\nabla^{i}{\rm
Rm}\ast\nabla^{j}{\rm Rm}+\sum_{i+j=\ell+2}\nabla^{i}H\ast\nabla^{j}H \nonumber\\
&&+ \ \sum_{i+j=\ell+2}\nabla^{i}F\ast\nabla^{j}F+\sum_{i+j+k=\ell}\nabla^{i}H\ast\nabla^{j}H\ast\nabla^{k}{\rm
Rm}\label{6.5}\\
&&+ \ \sum_{i+j+k=\ell}\nabla^{i}F\ast\nabla^{j}F\ast\nabla^{k}{\rm Rm}.\nonumber
\end{eqnarray}
In particular,
\begin{eqnarray*}
\frac{\partial}{\partial t}|\nabla^{l}{\rm
Rm}|^{2}&\leq&\Delta|\nabla^{l}{\rm Rm}|^{2}-2|\nabla^{\ell+1}{\rm
Rm}|^{2} \\
&&+ \ C\cdot\sum_{i+j=\ell}|\nabla^{i}{\rm Rm}|\cdot|\nabla^{j}{\rm
Rm}|\cdot|\nabla^{\ell}{\rm Rm}| \\
&&+ \ C\cdot\sum_{i+j=\ell+2}|\nabla^{i}H|\cdot|\nabla^{j}H|\cdot|\nabla^{\ell}{\rm
Rm}| \\
&&+ \ C\cdot\sum_{i+j=\ell+2}|\nabla^{i}F|\cdot|\nabla^{j}F|\cdot|\nabla^{\ell}{\rm
Rm}| \\
&&+ \ C\cdot\sum_{i+j+k=\ell}|\nabla^{i}H|\cdot|\nabla^{j}H|\cdot|\nabla^{k}{\rm
Rm}|\cdot|\nabla^{\ell}{\rm Rm}| \\
&&+ \ C\cdot\sum_{i+j+k=\ell}|\nabla^{i}F|\cdot|\nabla^{j}F|\cdot|\nabla^{k}{\rm
Rm}|\cdot|\nabla^{\ell}{\rm Rm}|.
\end{eqnarray*}
\end{proposition}

Since $\frac{\partial}{\partial t}F=\Delta F+{\rm Rm}\ast F$ it
follows that
\begin{eqnarray*}
\frac{\partial}{\partial t}\nabla F&=&\nabla\frac{\partial}{\partial
t}F+F\ast\nabla({\rm Rm}+H\ast H+F\ast F) \\
&=&\nabla(\Delta F+{\rm Rm}\ast F)+F\ast\nabla{\rm Rm}+F\ast
H\ast\nabla H+F\ast F\ast\nabla F \\
&=&\Delta(\nabla F)+\nabla{\rm Rm}\ast F+{\rm Rm}\ast\nabla F+F\ast H\ast\nabla H+F\ast F\ast\nabla F.
\end{eqnarray*}
It can be expressed as
\begin{eqnarray*}
\frac{\partial}{\partial t}\nabla F&=&\Delta(\nabla
F)+\sum_{i+j=1}\nabla^{i}F\ast\nabla^{j}{\rm Rm} \\
&&+ \ \sum_{i+j+k=1}\nabla^{i}F\ast\nabla^{j}F\ast\nabla^{k}F+\sum^{1-1}_{i=0}\sum^{1-i}_{j=0}\nabla^{i}F\ast\nabla^{j}H\ast\nabla^{1-i-j}H.
\end{eqnarray*}

More generally, we can show that

\begin{proposition} \label{p6.2}For ${\rm RF}(A,B)$ and any positive
integer $\ell$ we
have
\begin{eqnarray*}
\frac{\partial}{\partial
t}\nabla^{\ell}F&=&\Delta(\nabla^{\ell}F)+\sum_{i+j=\ell}\nabla^{i}F\ast\nabla^{j}{\rm
Rm} \\
&&+ \ \sum_{i+j+k=\ell}\nabla^{i}F\ast\nabla^{j}F\ast\nabla^{k}F
+\sum^{\ell-1}_{i=0}\sum^{\ell-i}_{j=0}\nabla^{i}F\ast\nabla^{j}H\ast\nabla^{\ell-i-j}H.
\end{eqnarray*}
In particular,
\begin{eqnarray*}
\frac{\partial}{\partial
t}|\nabla^{\ell}F|^{2}&\leq&\Delta|\nabla^{\ell}F|^{2}
-2|\nabla^{\ell+1}F|^{2}+C\cdot\sum_{i+j=\ell}|\nabla^{i}F|\cdot|\nabla^{j}{\rm
Rm}|\cdot|\nabla^{\ell}F|
\\
&&+ \ C\cdot\sum_{i+j+k=\ell}|\nabla^{i}F|\cdot|\nabla^{j}F|\cdot|\nabla^{k}F|\cdot|\nabla^{l}F|
\\
&&+ \ C\cdot\sum^{\ell-1}_{i=0}\sum^{\ell-i}_{j=0}|\nabla^{i}F
|\cdot|\nabla^{j}H|\cdot|\nabla^{\ell-i-j}H|\cdot|\nabla^{\ell}F|.
\end{eqnarray*}
\end{proposition}

Similarly, we obtain

\begin{proposition} \label{p6.3}For ${\rm RF}(A,B)$ and any positive integer $l$ we
have
\begin{eqnarray*}
\frac{\partial}{\partial
t}\nabla^{\ell}H&=&\Delta(\nabla^{\ell}H)+\sum_{i+j=\ell}
\nabla^{i}H\ast\nabla^{j}{\rm
Rm} \\
&&+ \ \sum_{i+j+k=\ell}\nabla^{i}H\ast\nabla^{j}H\ast\nabla^{k}H
+\sum^{\ell-1}_{i=0}\sum^{\ell-i}_{j=0}\nabla^{i}H\ast\nabla^{j}
F\ast\nabla^{\ell-i-j}F.
\end{eqnarray*}
In particular,
\begin{eqnarray*}
\frac{\partial}{\partial
t}|\nabla^{\ell}H|^{2}&\leq&\Delta|\nabla^{\ell}H|^{2}
-2|\nabla^{\ell+1}H|^{2}+C\cdot\sum_{i+j=\ell}|\nabla^{i}H
|\cdot|\nabla^{j}{\rm
Rm}|\cdot|\nabla^{\ell}H|
\\
&&+ \ C\cdot\sum_{i+j+k=\ell}|\nabla^{i}H|\cdot|\nabla^{j}H
|\cdot|\nabla^{k}H|\cdot|\nabla^{\ell}H|
\\
&&+ \ C\cdot\sum^{\ell-1}_{i=0}\sum^{\ell-i}_{j=0}|\nabla^{i}H
|\cdot|\nabla^{j}F|\cdot|\nabla^{\ell-i-j}F|\cdot|\nabla^{\ell}H|.
\end{eqnarray*}
\end{proposition}

From the evolution inequalities
\begin{eqnarray*}
\frac{\partial}{\partial t}|H|^{2}&\leq&\Delta|H|^{2}-2|\nabla
H|^{2}+C\cdot|{\rm Rm}|\cdot|H|^{2}, \\
\frac{\partial}{\partial t}|F|^{2}&\leq&\Delta|F|^{2}-2|\nabla
F|^{2}+C\cdot|{\rm Rm}|\cdot|F|^{2},
\end{eqnarray*}
the following theorem is obvious.

\begin{theorem} \label{t6.4}Suppose that $(g(x,t),H(x,t),F(x,t))$ is a solution to
${\rm RF}(A,B)$ on a compact manifold $M^{n}$ for a short time $0\leq
t\leq T$ and $K_{1}, K_{2}, K_{3}$ are arbitrary given nonnegative
constants. Then there exists a constant $C_{n}$ depending only on
$n$ such that if
\begin{equation*}
|{\rm Rm}(x,t)|_{g(x,t)}\leq K_{1}, \ \ \ |H(x)|_{g(x)}\leq K_{2}, \
\ \ |F(x)|_{g(x)}\leq K_{3}
\end{equation*}
for all $x\in M$ and $t\in[0,T]$, then
\begin{equation}
|H(x,t)|_{g(x,t)}\leq K_{2}e^{C_{n}K_{1}t}, \ \ \
|F(x,t)|_{g(x,t)}\leq K_{3}e^{C_{n}K_{1}t},\label{6.6}
\end{equation}
for all $x\in M$ and $t\in[0,T]$.
\end{theorem}

Parallelling to Theorem \ref{t4.6}, we can prove

\begin{theorem} \label{t6.5}Suppose that $(g(x,t),H(x,t),F(x,t))$ is a solution to
${\rm RF}(A,B)$ on a compact manifold $M^{n}$ and $K$ is an arbitrary
given positive constant. Then for each $\alpha>0$ and each integer
$m\geq1$ there exists a constant $C_{m}$ depending on
$m,n,\max\{\alpha,1\}$, and $K$ such that if
\begin{equation*}
|{\rm Rm}(x,t)|_{g(x,t)}\leq K, \ \ \ |H(x)|_{g(x)}\leq K, \ \ \
|F(x)|_{g(x)}\leq K
\end{equation*}
for all $x\in M$ and $t\in[0,\alpha/K]$, then
\begin{equation}
|\nabla^{m-1}{\rm
Rm}(x,t)|_{g(x,t)}+|\nabla^{m}H(x,t)|_{g(x,t)}
+|\nabla^{m}F(x,t)|_{g(x,t)}\leq\frac{C_{m}}{t^{\frac{m}{2}}},\label{6.7}
\end{equation}
for all $x\in M$ and $t\in(0,\alpha/K]$.
\end{theorem}

We can also establish the corresponding compactness theorem for
${\rm RF}(A,B)$. We omit the detail since the proof
is close to the proof in Section \ref{5}. In the forthcoming paper, we will
consider the BBS estimates for complete noncompact Riemmanian
manifolds.



\end{document}